\documentclass[12pt,a4paper]{article}

\usepackage[margin=1in]{geometry} 
\usepackage[utf8]{inputenc} 
\usepackage{authblk} 
\usepackage{hyperref} 
\usepackage[numbers,sort]{natbib} 
\usepackage{color}
\usepackage{array,multirow}

\usepackage{enumitem}
\setlist{itemsep=0pt}

\usepackage{amssymb,amsmath,amsthm}
\usepackage{mathtools}
\usepackage{graphicx}
\usepackage{subcaption}
\usepackage{algorithm, algorithmic}
\usepackage[capitalise]{cleveref}
\usepackage{ dsfont }

\theoremstyle{definition}
\newtheorem{theorem}{Theorem}[section]
\newtheorem{proposition}{Proposition}

\newtheorem{corollary}{Corollary}

\newtheorem{example}{Example}

\theoremstyle{remark}

\newcommand{\R}{\mathbb{R}}
\newcommand{\E}{\mathbb{E}}

\DeclareMathOperator{\trace}{tr}
\DeclareMathOperator{\Var}{Var}
\DeclareMathOperator{\Cov}{Cov}

\renewcommand{\d}{\mathrm{d}}
\renewcommand{\div}{\mathrm{div}}

\usepackage{xargs}                      
\usepackage[pdftex,dvipsnames]{xcolor}  
\usepackage[colorinlistoftodos,prependcaption,textsize=tiny,textwidth=0.85in]{todonotes}
\newcommandx{\oz}[2][1=]{\todo[linecolor=OliveGreen,backgroundcolor=OliveGreen!25,bordercolor=OliveGreen,#1]{#2}}

\begin{document}

\title{Optimal Riemannian metric for Poincaré inequalities and how to ideally precondition Langevin dymanics}

\author{
Tiangang Cui\footnote{tiangang.cui@sydney.edu.au},
Xin Tong\footnote{xin.t.tong@nus.edu.sg},
Olivier Zahm\footnote{olivier.zahm@inria.fr}
}

\maketitle

\begin{abstract}
Poincaré inequality is a fundamental property that rises naturally in different branches of mathematics. The associated Poincaré constant plays a central role in many applications since it governs the convergence of various practical algorithms. For instance, the convergence rate of the Langevin dynamics is exactly given by the Poincaré constant. This paper investigates a Riemannian version of Poincaré inequality where a positive definite weighting matrix field (\emph{i.e.} a Riemannian metric) is introduced to improve the Poincaré constant, and therefore the performances of the associated algorithm. Assuming the underlying measure is a \emph{moment measure}, we show that an optimal metric exists and the resulting Poincaré constant is 1. We demonstrate that such optimal metric is necessarily a \emph{Stein kernel}, offering a novel perspective on these complex but central mathematical objects that are hard to obtain in practice. We further discuss how to numerically obtain the optimal metric by deriving an implementable optimization algorithm. The resulting method is illustrated in a few simple but nontrivial examples, where solutions are revealed to be rather sophisticated. We also demonstrate how to design efficient Langevin-based sampling schemes by utilizing the precomputed optimal metric as a preconditioner.
~\\

\end{abstract}

\section{Introduction}

In 1890, Poincaré inequality was first proposed in \cite{poincare1890equations} as a result of calculus. Later, its deep connection with various directions of mathematics, such as functional analysis, geometry, and probability theory have been discovered and thoroughly studied \cite{bakry2008simple,bebendorf2003note,cheeger1999differentiability,ledoux2015stein,villani2021topics}.
Recently, the significance of Poincaré in applied mathematics have also been recognized in a wide range of applications. For examples, the convergence analysis of stochastic differential equations \cite{bakry2014analysis,vempala2019rapid}, MCMC samplers \cite{andrieu2022poincar,andrieu2022comparison,cheng2024fast,chewi2024analysis}, sensitivity analysis \cite{flock2023certified,roustant2017poincare,song2019derivative} and dimension reduction \cite{cui2022prior,li2023principal,li2024sharp,Teixeira2020,verdiere2023diffeomorphism}, to name just a few.
Given a probability measure $\mu$ on $\R^d$, the Poincaré constant $C(\mu)\in [0,\infty)$ is the smallest constant such that the Poincaré inequality
\begin{equation}\label{eq:Poincare}
 \Var_\mu(f) \leq C(\mu) \int \|\nabla f(x)\|^2 \d\mu(x)
\end{equation}
holds for any sufficiently smooth function $f:\R^d\rightarrow\R$.
Here, $\Var_\mu(f)=\int (f-\int f\d\mu)^2\d\mu$ denotes the variance of $f$ under $\mu$ and $\|\cdot\|$ is the Euclidean norm of $\R^d$. Showing that the Poincaré constant $C(\mu)$ is finite, and providing a tractable estimate of its optimal value, has long been a challenging and active problem in pure and applied mathematics \cite{acosta2004optimal,payne1960optimal,pillaud2020statistical}.

The Poincaré constant plays a central role in determining the performance and the convergence rate of numerous stochastic algorithms.
A classical example is the (overdamped) Langevin stochastic differential equation (SDE)
\begin{equation}\label{eq:SDE_classic}
 \d X_t = -\nabla V(X_t)\d t + \sqrt{2}\d B_t ,
\end{equation}
where $B_t$ is the standard Brownian motion in $\R^d$ and $V:\R^d\rightarrow\R$ is a smooth potential function. Denoting by $\mu_t$ the probability measure of the solution $X_t$ to \eqref{eq:SDE_classic}, it is well known that $\mu_t$ converges to the equilibrium distribution
\begin{equation}\label{eq:Vpotential}
 \d \mu (x) \propto\exp(-V(x))\d x ,
\end{equation}
if and only if $\mu$ satisfies the Poincaré inequality with $C(\mu)<\infty$, for instance, see \cite{bakry2014analysis} and Proposition \ref{prop:EquivalentSDE_WPI} below. In particular, $\mu_t$ converges to $\mu$ at a rate of $e^{-2t/C(\mu)}$, which highlights the central role of the Poincaré constant.

The Bakry-\'Emery theorem \cite{bakry2006diffusions,bakry2014analysis} provides the baseline estimate for the Poincaré constant.
It states that $C(\mu)\leq1/\rho$ for any measure $\d \mu(x)\propto\exp(-V(x))\d x$ with $\rho$-convex potential $V$, \emph{i.e.} $V$ is twice-differentiable and its Hessian is uniformly bounded as $\text{Hess}(V(x))\succeq \rho I_d$ for some constant $\rho>0$. Here, $\succeq$ denotes the Loewner order and $I_d \in \R^{d \times d}$ the identity matrix. However, this convexity assumption may not be satisfied by many probability measures of interest, such as those of multi-modal or heavy-tailed distributions. In these situations, the Poincaré constant can be arbitrarily large, or it may even be infinite.

In this paper, we consider a Riemannian version of the Poincaré inequality:
\begin{equation}\label{eq:Poincare_Riemannian}
 \Var_\mu(f) \leq C(\mu , W) \int  \nabla f(x)^\top  W(x)\nabla f(x) \d\mu(x) ,
\end{equation}
where $W:x\mapsto W(x)\in \mathcal{S}_{+}^d$ is a field taking value in the convex cone $\mathcal{S}_{+}^d\subset \R^{d\times d}$ of symmetric positive semi-definite matrices.
Again, $C(\mu, W)$ is the smallest constant such that \eqref{eq:Poincare_Riemannian} holds for any sufficiently smooth function $f$. The field $W$ provides a Riemannian metric on the space $\R^d$ that assigns to each $x\in\R^d$ an inner product $\langle u,v\rangle_x = u^\top W(x) v$, where $u,v\in\R^d$.
The inequality \eqref{eq:Poincare_Riemannian} extends existing work in several ways.
{
\begin{itemize}
 \item It encompasses the \emph{Brascamp--Lieb inequality} \cite{bobkov2000brunn,brascamp1976extensions}
 which corresponds to \eqref{eq:Poincare_Riemannian} with $W(x) = (\text{Hess\,}V(x))^{-1}$ and $C(\mu,W)=1$, assuming the function $V$ in \eqref{eq:Vpotential} is strongly convex. This log-concavity assumption on $\mu$ is, however, quite restrictive.

 \item The \emph{Mirror Poincaré inequality}, as defined in \cite{chewi2020exponential}, corresponds to $W(x) = (\text{Hess\,}\varphi(x))^{-1}$ for some strongly convex function $\varphi$ which is arbitrarily chosen based on the application \cite{liu2023mirror}.  Although the resulting constant $C(\mu , W)$ in \eqref{eq:Poincare_Riemannian} depends heavily on the choice of $\varphi$, there is currently no established notion of an optimal choice for $\varphi$.

 \item It also generalizes the \emph{weighted Poincaré inequalities} considered in \cite{bobkov2009weighted,bonnefont2016spectral}, in which $W(x)=w(x) I_d$ with a positive scalar field $w:x \rightarrow w(x) \in \R_{+}$. Such weighted Poincaré inequalities have been thoroughly analyzed in \cite{bonnefont2016note,saumard2019weighted,song2019derivative,heredia2024one} in dimension $d=1$.
\end{itemize}

}

The inequality \eqref{eq:Poincare_Riemannian} can be used in many ways. We motivate the work presented here by considering the case where the metric $W$ is used to precondition the Langevin dynamic \eqref{eq:SDE_classic}:
\begin{equation}\label{eq:SDE_matrixWeighted}
 \d X_t = (\div(W)- W\nabla V )\d t + \sqrt{2 W}\d B_t,
\end{equation}
where $\div(W) = (\sum_{j=1}^d \partial_i W_{i,j})_{1\leq i\leq d}$ is the divergence field of $W$ and $\sqrt{W}$ is any square root of $W$ such that $\sqrt{W}(x)\sqrt{W}(x)^\top  = W(x)$.
Following the discussion in \cite[Section 1.15.2]{bakry2014analysis}, such preconditioning strategy can be interpreted as a \emph{local} and \emph{anisotropic} time change of the SDE \eqref{eq:SDE_classic}. It is local because $W(X_t)$ may vary in space and anisotropic as the matrix $W(X_t)$ may not be proportional to the identity matrix. As a result, the diffusion term $\sqrt{2 W}\d B_t$ can locally accelerate or decelerate the particle $X_t$ in different directions. The preconditioning in the form of \eqref{eq:SDE_matrixWeighted} is referred to as the Riemannian Langevin dynamics, see \cite{girolami2011riemann,li2022mirror,patterson2013stochastic} for applications.
%
%
{The following proposition is a reformulation of Theorem 4.2.5 from \cite{bakry2014analysis}, which establishes the exponential convergence rate of \eqref{eq:SDE_matrixWeighted} toward the equilibrium measure $\mu$, provided that $\mu$ satisfies the Poincaré inequality \eqref{eq:Poincare_Riemannian}.
For completeness, the proof is given in Appendix \ref{proof:EquivalentSDE_WPI}.
}

\begin{proposition}\label{prop:EquivalentSDE_WPI}
 Let $\mu$ be a probability measure on $\R^d$ such that $\d\mu(x)\propto\exp(-V(x))\d x$ and let $W:\R^d\rightarrow\mathcal{S}_{+}^d$.
 Then for any $C\geq0$ the following assertions are equivalent.
 \begin{enumerate}
  \item $\mu$ satisfies the Riemannian Poincaré inequality \eqref{eq:Poincare_Riemannian} with $C(\mu,W)\leq C$.

  \item For any probability measure $\mu_0$, the solution $X_t\sim\mu_t$ to the Riemannian Langevin dynamics \eqref{eq:SDE_matrixWeighted} with $X_0\sim\mu_0$ satisfies
  \begin{equation}\label{eq:EquivalentSDE_WPI}
   \chi^2(\mu_t,\mu)\leq e^{-2t/C} \chi^2(\mu_0,\mu) ,
  \end{equation}
  for all $t\geq0$, where $\chi^2(\mu_t,\mu)=\Var_\mu(\d\mu_t/\d\mu)$ denotes the chi-square divergence.
 \end{enumerate}
\end{proposition}
In this setting, the Riemannian metric $W$ offers great flexibility in improving the convergence rate of the SDE \eqref{eq:SDE_matrixWeighted} by directly minimizing the Poincaré constant $C(\mu,W)$ over $W$.
In this paper, we analyze the optimization problem
\begin{equation}\label{eq:minC}
 \min_{\substack{ W: \R^d\rightarrow \mathcal{S}_{+}^d \\ \int \trace(W) \d\mu = \trace(\Cov_\mu)}} C(\mu,W) ,
\end{equation}
for dimensions $d\geq1$, where $\trace(\cdot)$ denotes the trace operator and $\Cov_\mu = \int (x-m)(x-m)^\top \d\mu(x) $ and $m=\int x\d\mu$ are respectively the covariance matrix and the mean of $\mu$, assuming they both exist. The normalization $\int \trace(W) \d\mu = \trace(\Cov_\mu)$ in \eqref{eq:minC} removes the trivial invariance of the Poincaré constant $C(\mu , \alpha W) = C(\mu , W)/\alpha $ for any positive scalar $\alpha>0$.
In addition, this normalization is convenient because it yields\footnote{Evaluating the Poincaré inequality \eqref{eq:Poincare_Riemannian} with the affine function $f(x)=x_i$ and summing over $i$ yields $\trace(\Cov_\mu) \leq C(\mu,W)\trace(\int W\d\mu)$ and then \eqref{eq:boundOnC}.} the lower bound
\begin{equation}\label{eq:boundOnC}
  C(\mu,W) \geq 1 .
\end{equation}

{
By the time we completed this work, we became aware of the recent contributions \cite{lelievre2024optimizing} addressing problems similar to \eqref{eq:minC}. Motivated by molecular dynamics simulations, the authors consider a probability distribution $\mu$ defined on the torus $\mathbb{T}^d=(\R\backslash\mathbb{Z})^d$ to account for the periodic boundary conditions of molecular systems. Additionally, instead of the normalization $\int \trace(W) \d\mu = \trace(\Cov_\mu)$ in \eqref{eq:EquivalentSDE_WPI}, they impose an $L^p$-type constraint $\int \|W(x)\|_F^p e^{-p V(x)} \d x \leq 1$ for some $p\geq1$. These differences result in significantly different approaches and solution algorithms.
}

Our contributions are the following. {First, in Section \ref{sec:ExitenceViaMomentMap}, we show that the optimization problem \eqref{eq:minC} admits a closed-form solution using the notion of the moment measure \cite{cordero2015moment}. This solution is given by
\begin{equation}
 W(x)=\text{Hess\,}\varphi( \nabla \varphi^{-1}(x)) ,
\end{equation}
for some strongly convex function $\varphi$ that depends on $\mu$. While such objects have been studied in the functional analysis literature \cite{klartag2014logarithmically,kolesnikov2016riemannian,klartag2013poincare}, their practical computation remains an open problem, which we address in the present work.
}
Then, in Section \ref{sec:OptimalMatrixField}, we provide characterizations and properties of the optimal metric $W(x)$, showing that they are necessarily symmetric positive-semidefinite \emph{Stein kernels} and relating them to the central limit theorem, see \cite{cacoullos1994variational,stein1986approximate}.
{Remarkably, $W$ being a Stein Kernel simplifies the preconditioned Langevin dynamics \eqref{eq:SDE_matrixWeighted} as follow
\begin{equation}
\d X_t = -(X_t-m) \d t + \sqrt{2W(X_t)} \d B_t .
\end{equation}
Such \emph{gradient-free} dynamics which does not envolve the term $\nabla V$ are receiving growing attention, see \cite{engquist2024adaptive,engquist2024sampling}.
}
Next, we show in Section \ref{sec:Convex} that \eqref{eq:minC} can be formulated as a concave optimization problem on the spectrum of the diffusion operator associated with the Riemannian Langevin dynamics \eqref{eq:SDE_matrixWeighted}.
In Section \ref{sec:NumericalSolution}, we propose a gradient-based algorithm to numerically solve \eqref{eq:minC} using the finite element method.
In Section \ref{sec:Applications}, we demonstrate the proposed algorithms on four benchmarks of dimension $d=2$. In particular, we show the benefit of using the preconditioned Langevin dynamic \eqref{eq:SDE_matrixWeighted} in the numerical discretization.

\section{Existence of optimal metric}\label{sec:ExitenceViaMomentMap}

The existence of a solution to \eqref{eq:minC} can be established using \emph{moment measures} \cite{cordero2015moment,santambrogio2016dealing}.
A probability measure $\mu$ is said to be the moment measure of a local Lipschitz convex function $\varphi : \R^d \rightarrow \R \cup\{+\infty\}$ if $\mu$ is the pushforward measure of $e^{-\varphi(z)}\d z$ under the transformation $z\mapsto \nabla \varphi(z)$, meaning
\begin{equation}\label{eq:MomentMap}
  \int f(x) \d\mu(x) = \int f(\nabla \varphi(z)) e^{-\varphi(z)} \d z ,
\end{equation}
for any $\mu$-integrable function $f$.
The convexity of $\varphi$ ensures that the map $z\mapsto \nabla \varphi(z)$ is invertible and the measure $e^{-\varphi(z)}\d z$ is log-concave.
From the perspective of the optimal transport theory, the invertible map $\nabla \varphi$ is the Brenier map that pushes forward the source measure $e^{-\varphi(z)}\d z$ to $\mu$ and uniquely minimizes the quadratic cost $\int \|z-\nabla \varphi(z)\|^2 e^{-\varphi(z)}\d z$.

Theorem 2 in \cite{cordero2015moment} shows that any probability measure $\mu$ is a moment measure if it is not supported on a hyperplane, and satisfies $\int |x|\d\mu<\infty$ and $\int x\d\mu=0$. Moreover, the corresponding \emph{moment map} $\varphi$ is unique up to translation.
The centred assumption on $\mu$ is a mild one: if $m=\int x\d\mu\neq 0$, it is sufficient to replace $\mu$ with the translated measure $\d\mu'(x)=\d\mu(x-m)$ to satisfy this assumption.

\begin{theorem}[Existence of a solution to \eqref{eq:minC}]\label{th:WstartMomentMap}
 Assume that $\mu$ is the moment measure of a strictly convex function $\varphi$ which is twice continuously differentiable. Then the field of symmetric positive-semidefinite matrices
 \begin{equation}\label{eq:WstartMomentMap}
  W(x) = \text{Hess}\,\varphi ( (\nabla \varphi)^{-1}(x) ),
 \end{equation}
 is a solution to \eqref{eq:minC} with $C(\mu,W)=1$.
\end{theorem}

\begin{proof}
 We recall that the Brascamp--Lieb inequality \cite{brascamp1976extensions} states that any log-concave measure $\d \nu(z) = e^{-\varphi(z)}\d z$ satisfies $\Var_{\nu}(f) \leq \int  \nabla f^\top  (\text{Hess}\,\varphi)^{-1} \nabla f \d\nu$ for any sufficiently smooth function $f$.
 Following Section 4 in \cite{fathi2019stein}, we can write
 \begin{align*}
  \Var_\mu(f)
  &= \Var_{e^{-\varphi}}(f\circ\nabla\varphi) \\
  &\leq \int (\nabla (f\circ\nabla\varphi) )^\top  \left(\text{Hess}\,\varphi \right)^{-1} \nabla (f\circ\nabla\varphi) e^{-\varphi}\d z\\
  &= \int (\nabla f\circ\nabla\varphi)^\top  \left(\text{Hess}\,\varphi \right)  (\nabla f\circ\nabla\varphi) e^{-\varphi}\d z\\
  &= \int \nabla f^\top  \left(\text{Hess}\,\varphi \circ (\nabla\varphi)^{-1}\right)  \nabla f \d\mu .
 \end{align*}
 Thus, $C(\mu,W)=1$ for $W$ defined in \eqref{eq:WstartMomentMap}.
 Next, we show that $W$ satisfies $\int \trace(W)\d\mu=\trace(\Cov_\mu)$.
 Theorem 2.3 in \cite{fathi2019stein} ensures that $W$ defined in \eqref{eq:WstartMomentMap} is a \emph{Stein kernel}, meaning that $\int xf\d\mu = \int W\nabla f\d\mu$ holds for any $f$.
 Letting $f$ go over all affine functions, we deduce that $\int W\d\mu = \int xx^\top \d\mu = \Cov_\mu$, where we used the fact that, because $\mu$ is a moment measure, it is necessarily centred (see Proposition 1 in \cite{cordero2015moment}). This yields $\int \trace(W)\d\mu=\trace(\Cov_\mu)$ so that, by \eqref{eq:boundOnC} and $C(\mu,W)=1$, we have that $W$ is a solution to \eqref{eq:minC}.
\end{proof}

Theorem \ref{th:WstartMomentMap} ensures the existence (but not the uniqueness) of solutions to the optimization problem \eqref{eq:minC} provided the moment map $\varphi$ is sufficiently smooth. This moment map is the solution to the toric Kähler-Einstein equation
\begin{equation}\label{eq:Kähler-Einstein}
 e^{-\varphi} = \mu(\nabla \varphi) \det(\text{Hess}\,\varphi),
\end{equation}
which can be obtained by applying the change of variable $x=\nabla\varphi(z)$ to the left-hand side of \eqref{eq:MomentMap}.
Note that \eqref{eq:Kähler-Einstein} is a variant of the Monge-Ampère equation \cite{figalli2017monge}.
As shown in \cite{berman2013real}, if $\mu$ is centred, supported on a compact and convex set and has density bounded from above and below by positive constants, then $\varphi$ is smooth enough so that Theorem \ref{th:WstartMomentMap} applies.

In principle, Theorem \ref{th:WstartMomentMap} offers a constructive way to build an optimal metric $W$: one can first numerically solve \eqref{eq:Kähler-Einstein} for $\varphi$ and then compute $W$ via \eqref{eq:WstartMomentMap}. However, the numerical solution to \eqref{eq:Kähler-Einstein} is not an easy task in general, mostly because of the convexity constraint of $\varphi$.
Later on in Section \ref{sec:NumericalSolution}, we propose an alternative way to compute $W$ which does not require solving \eqref{eq:Kähler-Einstein} for $\varphi$. Nonetheless, in the particular case where $\mu$ is a product measure, we have a closed-form expression for the optimal $W$.

\begin{proposition}\label{prop:ProductMeasure}
 Let $\mu=\mu_1\otimes\hdots\otimes \mu_L$ be a product probability measure on $\R^d$ where each $\mu_\ell$ is a marginal probability measure on $\R^{d_\ell}$ with $d=\sum_{\ell=1}^L d_\ell$. We denote $\tau_\ell = (1+\sum_{i=1}^{\ell-1}d_i, \ldots, \sum_{i=1}^{\ell}d_i)$ a multi-index with cadinality $d_\ell$ and $x_{\tau_\ell}$ the vector containing the coordinates of $x$ associated with $\mu_\ell$.
 Assume that, for any $1\leq \ell\leq L$, there exists a symmetric positve-semidefinite matrix field $W_\ell:\R^{d_\ell}\rightarrow\mathcal{S}_{+}^{d_\ell}$ such that $C(\mu_\ell,W_\ell)=1$ with $\int\trace(W_\ell)\d\mu_\ell=\trace(\Cov_{\mu_\ell})$. Then, the block-diagonal metric
 \begin{equation}\label{eq:ProductMeasure}
  W(x) =
 \begin{pmatrix}
  W_1(x_{\tau_1}) & & 0 \\
  &\ddots & \\
  0 &&W_L(x_{\tau_L})
 \end{pmatrix},
 \end{equation}
 is a solution to \eqref{eq:minC}. For any $\tau_\ell$ with cardinality one, i.e., $d_\ell=1$, suppose the corresponding $\mu_\ell$ has connected support on $\R$ with $\Cov_{\mu_{\ell}} = \int(x_{\tau_\ell}-m_{\tau_\ell})^2\d\mu <\infty$, then we have the closed-form expression
 \begin{equation}\label{eq:SteinKernel1D}
  W_\ell(x_{\tau_\ell}) = \frac{1}{\mu_{\tau_\ell}(x_{\tau_\ell})}\int_{x_{\tau_\ell}}^\infty (t - m_{\tau_\ell})\d\mu_\ell(t) .
 \end{equation}
\end{proposition}

\begin{proof}
 Because $\int\trace(W_\ell)\d\mu_\ell=\trace(\Cov_{\mu_\ell})$, we have $\int \trace(W)\d\mu = \sum_{\ell=1}^L\trace(\Cov_{\mu_\ell}) =\trace(\Cov_\mu)$. By the stability of Poincaré inequalities under tensorization, e.g., see Proposition 4.3.1 in \cite{bakry2014analysis}, we have
 $C(\mu,W) = \max_\ell\{ C(\mu_\ell,W_{\ell})\} = 1$, because $C(\mu_\ell,W_\ell)=1$. This way, $W$ is a solution to \eqref{eq:minC}.
 Finally, for the case $d_\ell=1$, Theorem 6 in \cite{saumard2019weighted} ensures that $W_\ell$ defined in \eqref{eq:SteinKernel1D} satisfies $C(\mu_\ell,W_\ell)=1$.
 It remains to check that $\int\trace(W_\ell)\d\mu_\ell = \trace(\Cov_{\mu_\ell})$. Applying integration by part, we have
 $$
  \int W_\ell \d\mu_\ell = \int_{-\infty}^{+\infty} \left(\int_{x_{\tau_\ell}}^\infty (t - m_{\tau_\ell})\d\mu_\ell(t)\right)\d x_{\tau_\ell} = \int_{-\infty}^{+\infty} (x_{\tau_\ell} - m_{\tau_\ell})^2  \d\mu_\ell(x_{\tau_\ell}) = \Cov_{\mu_\ell},
 $$
 which concludes the proof.
\end{proof}

We give a few examples of measures where \eqref{eq:ProductMeasure} and \eqref{eq:SteinKernel1D} have closed-form expressions. We refer to \cite[Section 3.1]{saumard2019weighted} for in-depth discussion.
\begin{example}\label{ex:ProductMeasure}
For the standard Gaussian measure $\d\mu_\mathcal{N}(x)\propto\prod_{i=1}^d\exp(-|x_i|^2/2)\d x$ on $\R^d$, we have
$$
 W_\mathcal{N}(x) =
 \begin{pmatrix}
  1 & & 0 \\
  &\ddots & \\
  0 && 1
 \end{pmatrix}.
$$
For the Laplace measure $\d\mu_\mathcal{L}(x)\propto\prod_{i=1}^d\exp(-|x_i|)\d x$ on $\R^d$, we have
$$
\qquad
 W_\mathcal{L}(x) =
 \begin{pmatrix}
  1+|x_1| & & 0 \\
  &\ddots & \\
  0 && 1+|x_d|
 \end{pmatrix}.
$$
The generalized Cauchy distribution $\d\mu_\mathcal{C}\propto\prod_{i=1}^d(1+|x_i|^2)^{-\beta}\d x$ on $\R^d$ with parameter $\beta>1$ yields
$$
 W_\mathcal{C}(x) = \frac{1}{2(\beta-1)}
 \begin{pmatrix}
  1+|x_1|^2 & & 0 \\
  &\ddots & \\
  0 && 1+|x_d|^2
 \end{pmatrix}.
$$
\end{example}

The following corollary characterizes the optimal metric for sums of independent random variables.
\begin{corollary}
 Let $X,Y$ taking values in $\R^d$ be two independent random vectors and consider $Z=\alpha X+\beta Y$ for some $\alpha,\beta\in\R$.
 Assume there exists optimal metrics $W_X$ and $W_Y$ such that $C(\mu_X,W_X)=C(\mu_Y,W_Y)=1$, as well as $\int\trace(W_X)\d\mu_X=\trace(\Cov(X))$ and $\int\trace(W_Y)\d\mu_Y=\trace(\Cov(Y))$. Then, the conditional expectation
 $$
  W_{Z}(z) = \E\left[ \alpha^2 W_X(X) + \beta^2 W_Y(Y) | Z=z \right] ,
 $$
 is also an optimal metric for $Z$.
 In particular, for $\overline{X}_N=\frac{1}{\sqrt{N}}\sum_{i=1}^N X_i$ where $X_1\hdots,X_N$ are $N$ independent copies of $X$, we have
 \begin{equation}\label{eq:Wiid}
  W_{\overline{X}_N}(x) =  \E\left[\left. \frac{W_X(X_1) + \hdots + W_X(X_N)}{N} \right| \overline{X}_N = x \right] .
 \end{equation}
\end{corollary}
\begin{proof}
 Let $\mu=\mu_X\otimes\mu_Y$ be the measure of the random vector $(X,Y)$.
 By Proposition \ref{prop:ProductMeasure} we have an optimal metric  $\text{diag}(W_X(X),W_Y(Y))$ for $(X,Y)$.
 Thus, applying the Riemannian Poincaré inequality \eqref{eq:Poincare_Riemannian} with $f(x,y) = h(\alpha x + \beta y)$ yields
 \begin{align*}
  \Var(h(Z))
  &=  \Var(f(X,Y)) \\
  &\leq \E\left[ \nabla f(X,Y)^\top \text{diag}(W_X(X),W_Y(Y)) \nabla f(X,Y) \right] \\
  &= \E\left[ \alpha^2 \nabla h(Z)^\top W_X(X)\nabla h(Z)^\top + \beta^2  \nabla h(Z)^\top W_Y(Y)\nabla h(Z)^\top \right] \\
  &= \E\left[  \nabla h(Z)^\top \left( \alpha^2 W_X(X) + \beta^2  W_Y(Y) \right)\nabla h(Z)^\top \right].
 \end{align*}
 We conclude that $C(\mu_Z,W_Z)=1$. Furthermore, using the independence of $X$ and $Y$, we have $\int\trace(W_Z)\d\mu_Z = \alpha^2 \trace \Cov(X) + \beta^2 \Cov(Y) = \Cov(Z)$. We conclude that $W_Z$ is an optimal metric for $Z$. Then,  \eqref{eq:Wiid} follows by recursion and concludes the proof.
\end{proof}

\section{Characterization and properties of optimal metrics}\label{sec:OptimalMatrixField}

We start with a simple characterization of the solutions to the optimization problem \eqref{eq:minC} with $C(\mu,W)=1$.

\begin{theorem}\label{th:OptimalityCondition}
 Let $W:\R^d\rightarrow\mathcal{S}_+^d$ satisfy $\int \trace(W)\d\mu = \trace(\Cov_\mu)$. Then, $C(\mu,W)=1$ if and only if the Poincaré inequality \eqref{eq:Poincare_Riemannian} becomes an equality for all affine functions, that is
 \begin{equation}\label{eq:AffineFunctionsSaturatesPI}
  \Var_\mu(f) =  C(\mu,W) \int  \nabla f^\top W \nabla f\d\mu, \qquad \forall f: \text{ affine function}.
 \end{equation}
 In this case we have
  \begin{equation}\label{eq:intW=Cov}
   \int W\d\mu = \Cov_\mu .
  \end{equation}
\end{theorem}

\begin{proof}
 First, we assume that the identity in \eqref{eq:AffineFunctionsSaturatesPI} holds. Choosing $f(x) = x_i$ in \eqref{eq:AffineFunctionsSaturatesPI} and summing over $1\leq i \leq d$ yield $\trace(\Cov_\mu) = C(\mu,W)\trace(\int W\d\mu)$, and thus $C(\mu,W)=1$.
 Next, we assume $C(\mu,W)=1$ and show \eqref{eq:AffineFunctionsSaturatesPI} by contradiction. Assume there exists an affine function $f(x)=\alpha^\top x +\beta$ with $\alpha \in \R^d, \beta \in \R$ such that $\Var_\mu(f) <  \int  \nabla f^\top W \nabla f\d\mu$, which is equivalent to
 $
  \alpha^\top \Cov_\mu\alpha < \alpha^\top \left(\int  W \d\mu\right)\alpha.
 $
 Define $u_1=\alpha/\|\alpha\|$ and choose vectors $u_2,\dots,u_d$ to form a unitary matrix $U=[u_1,\hdots,u_d]\in\R^{d\times d}$. We have
 $$
  \trace(\Cov_\mu)
  = \sum_{i=1}^d u_i^\top  \Cov_\mu u_i
  <
  \sum_{i=1}^d u_i^\top \left( \int W\d\mu \right)u_i
  = \trace\left( \int W\d\mu \right),
 $$
 which contradicts $\int\trace(W)\d\mu=\trace(\Cov_\mu)$.
 Finally, the identity in \eqref{eq:intW=Cov} is a direct consequence of \eqref{eq:AffineFunctionsSaturatesPI} with $C(\mu,W)=1$, which concludes the proof.
\end{proof}

Next, we show that any solutions to \eqref{eq:minC} is necessarily a \textit{Stein kernel}.
A Stein kernel of a probability measure $\mu$ is a matrix field $W:\R^d\rightarrow\R^{d\times d}$ (not necessarily symmetric or positive-semidefinite) such that
\begin{equation}\label{eq:SteinKernel}
  \int (x-m) f \d\mu = \int W\nabla f \d\mu  ,
\end{equation}
for any smooth function $f:\R^d\rightarrow\R$, where $m=\int x\d\mu$ is the mean of $\mu$.
{
One notable property of $W$ being a Stein Kernel is that $\div(W) - W\nabla V = x-m$ (by performing one integration by parts on \eqref{eq:SteinKernel}). Consequently, the preconditioned Langevin dynamics \eqref{eq:SDE_matrixWeighted} simplifies as
\begin{equation}
 \d X_t \overset{\eqref{eq:SDE_matrixWeighted} \& \eqref{eq:SteinKernel}}{=} -(X_t-m) \d t + \sqrt{2W(X_t)} \d B_t .
\end{equation}
Remarkably, this dynamic is \emph{gradient-free} in the sense that it no longer involves the gradient of the potential $V$ associated with the measure $\mu$.
}

\begin{theorem}\label{th:Stein}
 If a symmetric positive-semidefinite matrix field $W$ is a solution to the optimization problem \eqref{eq:minC} with $C(\mu,W)=1$, then $W$ is a Stein kernel of $\mu$.
\end{theorem}

\begin{proof}
 By Theorem \ref{th:OptimalityCondition} we have $\Cov_\mu = \int W\d\mu$.
 Then, for any $f:\R^d\rightarrow\R$, $\alpha\in\R^d$ and $h\in\R$, the function $g=\alpha^\top  x + hf$ satisfies
 \begin{align*}
  \Var_\mu( g )&=\alpha^\top  \Cov_\mu \alpha
  +2h \alpha^\top\int (x-m)f\d\mu
  +h^2\Var_\mu(f), \\
  \int \nabla g^\top W \nabla g\d\mu &=
  \alpha^\top  \Cov_\mu \alpha
  + 2 h \alpha^\top  \int W \nabla f \d\mu
  +h^2 \int  \nabla f^\top  W \nabla f \d\mu .
 \end{align*}
 Applying the Riemannian Poincaré inequality \eqref{eq:Poincare_Riemannian} to $g$ yields
 \begin{align*}
  0\leq
  h \alpha^\top \left ( \int W \nabla f \d\mu -\int (x-m)f\d\mu\right)
  +\mathcal{O}(h^2) .
 \end{align*}
 For this to hold for any $h\in\R$ and $\alpha\in\R^d$, it is necessary that $\int W \nabla f \d\mu -\int (x-m)f\d\mu = 0$, which defines $W$ as a Stein kernel.
\end{proof}

Since the seminal work \cite{stein1986approximate}, Stein kernels have been implicitly used in numerous works on Stein's method, and they have received growing attention in several recent investigations, see \cite{courtade2019existence,fang2021high,chernozhuokov2022improved} for instance.
In particular, it is shown in \cite[Proposition 3.1]{fathi2019stein} that the so-called \textit{Stein discrepancy}
$$
 S_p(\mu|\gamma) = \inf_{\substack{W:\text{ Stein kernel}\\\text{ as in \eqref{eq:SteinKernel}}}} \left( \int \|W-I_d\|_F^p\d\mu\right)^{1/p},
$$
measures the divergence from $\mu$ to the Gaussian measure $\gamma=\mathcal{N}(0,I_d)$, which remarkably bounds the Wasserstein-$p$ distance between $\mu$ and $\gamma$ for $p\geq2$, see also \cite{ledoux2015stein} for the original proof for $p=2$.
A notable application of Stein kernels is that they often lead to non-asymptotic central limit theorems (CLT) \cite{barbour2005introduction,chen2021stein}. The following corollary provides such a CLT based on the optimal metric. {It is a variant of Theorem 3.3 in \cite{fathi2019stein}}, and its proof is given in Appendix \ref{proof:CLT}.

\begin{corollary}
\label{cor:CLT}
Let $X_1,\hdots,X_N$ be independent copies of $X\sim\mu$ such that $\E[X]=0$ and $\Cov(X)=I_d$, and let $\mu_N$ be the measure of $\overline{X}_N=\frac{1}{\sqrt{N}}\sum X_i$. Consider the Wasserstein-$p$ distance between $\mu_N$ and $\gamma=\mathcal{N}(0,I_d)$ given by
\[
\mathcal{W}_p(\mu_N,\gamma)=\left(\inf_{\pi\in\Pi(\mu_N,\gamma)} \int \|x-y\|^p\pi(\d x,\d y)\right)^{1/p},
\]
where $\Pi(\mu_N,\gamma)$ denotes the set of couplings between $\mu_N$ and $\gamma$.
Then, for any $p\geq 2$, we have
\[
\mathcal{W}_p(\mu_N,\gamma)\leq \frac{C_{p} d^{1-2/p}}{\sqrt{N}} \left(\int \| W(X)-I_d\|_p^{p} \d\mu \right)^{1/p},
\]
where $W$ is a solution to \eqref{eq:minC} with $C(\mu,W)=1$ and $C_p$ is a constant depending only on $p$ and where $\| \cdot \|_p$ is the matrix norm defined by $\|A\|_p^p = \sum_{ij} A_{ij}^p$.
\end{corollary}

Compared to Theorem 3.3 in \cite{fathi2019stein}, the main improvement of Corollary \ref{cor:CLT} is that it does not require $\mu$ to be log-concave. It applies to more general probability measures, for example, those heavy-tail distributions in Example \ref{ex:ProductMeasure}, as long as the optimal metric $W$ for the underlying moment measure $\mu$ satisfies $\E[ \|  W(X)-I_d  \|_F^p ] < \infty$ for the chosen $p$.

Another immediate property of Stein kernels is that they permit to derive a lower bound on the variance of a function. The following corollary is a generalization of the results of \cite{cacoullos1982upper} for dimension $d>1$.

\begin{corollary}
 Let a symmetric positive-semidefinite matrix field $W$ be a solution to \eqref{eq:minC} with $C(\mu,W)=1$.
 Then for any smooth function $f:\R^d\rightarrow\R$ we have
 \begin{equation}\label{eq:BoundOnVariance}
  \left\|\int W\nabla f \d\mu \right\|_{\Cov_\mu^{-1}}^2  \leq \Var_\mu(f) \leq \int \nabla f^\top  W\nabla f \d\mu ,
 \end{equation}
 where $\|v\|_{\Cov_\mu^{-1}}^2 = v^\top \Cov_\mu^{-1} v$.
\end{corollary}
\begin{proof}
 By Theorem \ref{th:Stein}, $W$ is a Stein kernel of $\mu$.
 The Cauchy--Schwarz inequality yields
 \begin{align*}
  \left|\alpha^\top  \int W\nabla f \d\mu \right|^2
  &\overset{\eqref{eq:SteinKernel}}{=} \left|\alpha^\top \int (x-m) \left(f-\int f\d\mu\right) \d\mu \right|^2
  \leq \left(\alpha^\top \Cov_\mu \alpha \right) \Var_\mu(f) ,
 \end{align*}
 for any $\alpha\in\R^d$. Taking the supremum over $\alpha\in\R^d$ such that $\alpha^\top \Cov_\mu \alpha=1$ yields the left-hand side of \eqref{eq:BoundOnVariance}. The right-hand side of \eqref{eq:BoundOnVariance} is just the Riemannian Poincaré inequality \eqref{eq:Poincare_Riemannian} with $C(\mu,W)=1$.
\end{proof}

We conclude this section with an excursion. Let $f$ be a 1-Lipschitz function so that $\|\nabla f(x)\|\leq1$. The right inequality of \eqref{eq:BoundOnVariance} yields $\Var_\mu(f)\leq \int \lambda_{\max}(W) \| \nabla f\|^2\d\mu \leq \int \lambda_{\max}(W) \d\mu $. Maximizing over the set of 1-Lipschitz function then yields
\begin{equation}\label{eq:BoundOfMaxVarianceOfLip}
  \sup_{\substack{f:\R^d\rightarrow\R \\ \text{1-Lipschitz} }} \Var_\mu(f)
  \leq \int \lambda_{\max}(W)\d\mu.
\end{equation}
This gives an upper bound on the maximal variance of $1$-Lipschitz function, which is known to be a difficult question related to the isoperimetric problem \cite{bobkov1996variance}.

\section{Convex analysis and spectral gap}\label{sec:Convex}

While Theorem \ref{th:OptimalityCondition} gives a necessary and sufficient condition for $C(\mu,W)=1$, it does not provide a practical way to build such an optimal metric $W$. In the following, we discuss how to numerically obtain $W$ using optimization methods. Toward this goal, we introduce the following Sobolev spaces
\begin{align*}
 L^2(\mu) &= \left\{f:\R^d\rightarrow\R: \|f\|^2_{L^2(\mu)} = \int f^2 d\mu <\infty  \right\}, \\
 H^1(\mu,W) &= \left\{f:\R^d\rightarrow\R: \|f\|^2_{H^1(\mu,W)} = \int f^2 \d\mu + \int \nabla f^\top W \nabla f\d\mu <\infty  \right\} .
\end{align*}
The space $H^1(\mu,W)$ corresponds to the largest function space for $f$ for which
the terms $\Var_\mu(f)$ and $\int \nabla f^\top  W \nabla f \d\mu$ of \eqref{eq:Poincare_Riemannian} are both finite. The Poincaré constant $C(\mu,W)$ can thus be expressed as
\begin{equation}\label{eq:inversePoincareConstant}
 \frac{1}{C(\mu,W)} = \inf_{\substack{f\in H^1(\mu,W) \\ \int f \d\mu = 0 }} \frac{\int \nabla f^\top  W \nabla f \d\mu}{\int f^2\d\mu} .
\end{equation}
The next proposition shows that the variational problem \eqref{eq:minC} can be solved by a convex optimization problem.

\begin{proposition}\label{prop:CisConvex}
 The function $W\mapsto C(\mu,W)^{-1}$ is concave over the convex cone of symmetric positive-semidefinite matrix field $W$.
 Furthermore, for any symmetric positive-semidefinite matrix fields $W$ and $W'$ such that $W'\preceq \Omega W$ for some constant $\Omega$, we have
 \begin{equation}\label{eq:optimalityOfW}
  \frac{1}{C(\mu,W')} \leq \frac{1}{C(\mu,W)} + \int\trace\left( (W'-W) G \right)\d\mu,
 \end{equation}
 for any $G\in \mathcal{G}(W)$ where
 \begin{equation}\label{eq:superGradient}
  \mathcal{G}(W) =
  \text{conv}\left\{
  \frac{\nabla u \nabla u^\top}{\int u^2\d\mu} :
   u \in \underset{\substack{ f\in H^1(\mu,W) \\ \int f \d\mu = 0  }}{\text{arg\,min}} \frac{\int \nabla f^\top  W \nabla f \d\mu}{\int f^2\d\mu}  \right\} .
 \end{equation}
 Here $\text{conv}\{\cdot\}$ denotes the convex hull of a set.

\end{proposition}

\begin{proof}
 For any symmetric positive-semidefinite matrix field $W_0,W_1$ and any $0<t<1$, the matrix field $W_t=tW_1+(1-t)W_0$ is a symmetric positive-semidefinite.
 By construction, for any $f\in H^1(\mu,W_t)$ we have
 \begin{align*}
  \|f\|_{H^1(\mu,W_t)}^2 = t \|f\|_{H^1(\mu,W_1)}^2 + (1-t)\|f\|_{H^1(\mu,W_0)}^2,
 \end{align*}
 which gives
 \[
 \|f\|_{H^1(\mu,W_1)}^2 \leq t^{-1}\|f\|_{H^1(\mu,W_t)}^2 \quad \text{and} \quad \|f\|_{H^1(\mu,W_0)}^2 \leq (1-t)^{-1}\|f\|_{H^1(\mu,W_t)}^2.
 \]
 We deduce that $H(\mu,W_t)\subseteq H(\mu,W_0)$ and $H(\mu,W_t)\subseteq H(\mu,W_1)$, which gives
 \begin{align*}
  \frac{1}{C(\mu,W_t)}
  &= \inf_{\substack{f\in H^1(\mu,W_t) \\ \int f \d\mu = 0 }} \frac{\int \nabla f^\top  W_t \nabla f \d\mu}{\int f^2\d\mu} \\
  &\geq t \inf_{\substack{f\in H^1(\mu,W_t) \\ \int f \d\mu = 0 }} \frac{\int \nabla f^\top  W_1 \nabla f \d\mu}{\int f^2\d\mu}
  + (1-t) \inf_{\substack{f\in H^1(\mu,W_t) \\ \int f \d\mu = 0 }} \frac{\int \nabla f^\top  W_0 \nabla f \d\mu}{\int f^2\d\mu}\\
  &\geq t \inf_{\substack{f\in H^1(\mu,W_1) \\ \int f \d\mu = 0 }} \frac{\int \nabla f^\top  W_1 \nabla f \d\mu}{\int f^2\d\mu}
  + (1-t) \inf_{\substack{f\in H^1(\mu,W_0) \\ \int f \d\mu = 0 }} \frac{\int \nabla f^\top  W_0 \nabla f \d\mu}{\int f^2\d\mu}\\
  &=  \frac{t}{C(\mu,W_1)} +  \frac{1-t}{C(\mu,W_0)}.
 \end{align*}
 Therefore, $W\mapsto C(\mu,W)^{-1}$ is concave.
 To show \eqref{eq:optimalityOfW}, we first notice that assumption $W'\preceq \Omega W$ implies
 $$
  \|f\|_{H^1(\mu,W')}^2
  \leq \int f^2 \d\mu + \Omega \int \nabla f^\top W\nabla f\d\mu
  \leq (1+\Omega) \|f\|_{H^1(\mu,W)}^2 ,
 $$
 so that $H^1(\mu,W) \subseteq H^1(\mu,W')$. Then, any function  $G$ in $\mathcal{G}(W)$ can be expressed as
 $$
  G = \sum_{i=1}^m \omega_i \frac{\nabla u_i \nabla u_i ^\top}{\int u_i^2\d\mu} ,
 $$
 for some integer $m\geq1$, some non-negative weights $\omega_i\geq 0$ with the normalisation $\sum_{i=1}^m\omega_i=1$, and
 \[
 u_i \in \mathrm{arg\,min}\left\{ \frac{\int \nabla f^\top  W \nabla f \d\mu}{\int f^2\d\mu}: f\in H^1(\mu,W) , \int f \d\mu = 0 \right\}.
 \]
 Because $u_i \in H^1(\mu,W) \subseteq H^1(\mu,W')$ we have
 \begin{align*}
  \frac{1}{C(\mu,W')}
  &= \inf_{\substack{f\in H^1(\mu,W') \\ \int f \d\mu = 0 }}
  \frac{\int \nabla f^\top  W \nabla f \d\mu}{\int f^2\d\mu} + \frac{\int \nabla f^\top  (W'-W) \nabla f \d\mu}{\int f^2\d\mu} \\
  &\leq \frac{\int \nabla u_i^\top  W \nabla u_i \d\mu}{\int u_i^2\d\mu}+ \frac{\int \nabla u_i^\top  (W'-W) \nabla u_i \d\mu}{\int u_i^2\d\mu} \\
  &= \frac{1}{C(\mu,W)}  + \int\trace\left( (W'-W) \frac{\nabla u_i \nabla u_i ^\top}{\int u_i^2\d\mu} \right)\d\mu.
 \end{align*}
 Multiplying the above inequality with $\omega_i\geq0$ and summing over $i=1,\hdots,m$ yields \eqref{eq:optimalityOfW} and concludes the proof.
\end{proof}

Let us comment on Proposition \ref{prop:CisConvex}.
First, because $W\mapsto C(\mu,W)^{-1}$ is concave, it admits a \emph{super-differential} $\partial_W(C(\mu,W)^{-1})$ at any $W$. The super-differential is the set of all $G$ such that \eqref{eq:optimalityOfW} holds for any $W'$, see \cite{boyd2004convex,rockafellar1970convex} for detailed explainations.
Here, the super-differential is identified using the scalar product $\langle W,G \rangle = \int \trace(WG)\d\mu$.
The set $\mathcal{G}(W)$ defined in \eqref{eq:superGradient} is thus a subset of $\partial_W(C(\mu,W)^{-1})$.
This will be used later on in Section \ref{sec:Gradient} to propose a gradient-based algorithm for solving \eqref{eq:minC}.
Second, the definition of $\mathcal{G}(W)$ in \eqref{eq:superGradient} requires that a minimizer of $ \frac{\int \nabla f^\top  W \nabla f \d\mu}{\int f^2\d\mu} $ over $f\in H^1(\mu,W)$ exists.
To analyze the existence of these minimizers, we introduce the linear operator $\mathcal{L}_\mu^W$ from the space $H^1(\mu,W)$ to its dual space, which is defined as
\begin{equation}\label{eq:LW}
 \int f \mathcal{L}_\mu^W g \d\mu = - \int \nabla f^\top  W \nabla g \d\mu ,
\end{equation}
for any functions $f,g\in H^1(\mu,W)$.
The operator $\mathcal{L}_\mu^W$ is the infinitesimal generator associated with the SDE \eqref{eq:SDE_matrixWeighted}.
Using the integration by part, the operator $\mathcal{L}_\mu^W$ can be expressed as
$$\mathcal{L}_\mu^W f = W :  \nabla^2 f  +  (\text{div}(W) + W \nabla \log\mu )^\top  \nabla f,$$ where $W:\nabla^2 f = \sum_{i,j=1}^dW_{i,j}\partial^2_{i,j}f$.
The following proposition relates $C(\mu,W)$ with the spectrum of $\mathcal{L}_\mu^W$, and ensures that minimizers of $ \frac{\int \nabla f^\top  W \nabla f \d\mu}{\int f^2\d\mu} $ over  $f\in H^1(\mu,W)$ exists.

\begin{proposition}\label{prop:EigendecompositionOfLW}
 Assume we have a symmetric positive-semidefinite matrix field $W$ such that
 the injection of $H^1(\mu,W)$ in $L^2(\mu)$ is compact and $H^1(\mu,W)$ is dense in $L^2(\mu)$.
 Then, if $C(\mu,W)<\infty$, there exists an orthonormal basis $\{u_i^W\}_{i=1}^\infty$ of $L^2(\mu)$ that forms an eigenbasis for $-\mathcal{L}_\mu^W$. That is,
 \begin{equation}\label{eq:EigendecompositionOfLW}
  -\mathcal{L}_\mu^W u_i^W = \lambda_i^W u_i^W ,
 \end{equation}
 where $\lambda_i^W\geq0$ is the $i$-th smallest eigenvalue of $-\mathcal{L}_\mu^W$. In addition, the following properties hold.
 \begin{enumerate}
  \item The eigenvalues $\lambda_i^W$ tends to $+\infty$ when $i\rightarrow +\infty$,
  \item The first eigenvalue is trivial $\lambda_1^W = 0$ and the associated eigenfunction is the constant function $u_1^W = \mathds{1}$,
  \item The second eigenvalue is positive $\lambda_2^W>0$ and satisfies
 \begin{equation}\label{eq:Lambda1}
  \lambda_2^W   = \frac{1}{C(\mu,W)}.
 \end{equation}
 \item The set of minimizers of $\frac{\int \nabla f^\top  W \nabla f \d\mu}{\int f^2\d\mu}$ over $f\in H^1(\mu,W)$ with $\int f\d\mu=0$, is the eigenspace $\text{span}\{u_2^W,\hdots,u_{m+1}^W\}$ associated with the eigenvalue $\lambda_2^W$, and its dimension $m$ is necessarily finite. Moreover, the super-differential of the concave function $W\mapsto \lambda_2^{W}$ is given by
 \begin{equation}\label{eq:superGradientOfLW}
  \partial_W \lambda_2^{W} =
  \overline{ \text{conv}  }  \left\{
  \frac{\nabla u \nabla u^\top}{\int u^2\d\mu} :
   u \in \text{span}\{u_2^W,\hdots,u_{m+1}^W\}
  \right\}  .
 \end{equation}
 where $\overline{(\cdot)}$ denotes the closure with respect to the topology induced by the $L^2(\mu;\R^{d\times d})$-norm on the matrix field $\|A\|_{L^2(\mu;\R^{d\times d})}=  (\int\|A\|_F^2\d\mu)^{1/2}$.
 In particular if $m=1$ then $\lambda_2^W$ is differentiable and
 \begin{equation}\label{eq:GradientOfLW}
  \nabla \lambda_2^W \overset{m=1}{=} \frac{ (\nabla u_2^W)(\nabla u_2^W)^\top}{\int (u_2^W)^2\d\mu}.
 \end{equation}
 \end{enumerate}

\end{proposition}

\begin{proof}
 To show that $-\mathcal{L}_\mu^W$ admits the eigendecomposition \eqref{eq:EigendecompositionOfLW}, it is sufficient to show that the bi-linear form $(f,g)\mapsto\int f (-\mathcal{L}_\mu^W)g \d\mu=\int \nabla f^\top W \nabla g\d\mu$ is continuous and coercive on $H^1(\mu,W)\subset L^2(\mu)$, see Theorem 7.3.2 in \cite{allaire2007numerical} for instance.
 To show the continuity, the Cauchy--Schwarz inequality leads to
 \begin{align*}
  &\left| \int f (-\mathcal{L}_\mu^W)g \d\mu - \int f' (-\mathcal{L}_\mu^W)g' \d\mu \right| \\
  &=\left| \int (\sqrt{W}\nabla f)^\top  (\sqrt{W}(\nabla g-\nabla g')) \d\mu - \int (\sqrt{W}(\nabla f'-\nabla f))^\top \sqrt{W} \nabla g' \d\mu \right| \\
  &\leq \|g-g'\|_{H^1(\mu,W)} \left(\int \nabla f^\top W \nabla f \d\mu\right)^{1/2} +\|f-f'\|_{H^1(\mu,W)} \left( \int \nabla g'^\top W \nabla g' \d\mu \right)^{1/2}.
 \end{align*}
 Then $\int f (-\mathcal{L}_\mu^W)g \d\mu \rightarrow \int f' (-\mathcal{L}_\mu^W)g' \d\mu$ whenever $f\rightarrow f'$ and $g\rightarrow g'$ in $H^1(\mu,W)$.
 The coercivity is a direct consequence of $C(\mu,W)<\infty$ since
 $$
  \|f\|^2_{H^1(\mu,W)} = \int f^2\d\mu + \int \nabla f^\top W \nabla f \d\mu \leq (C(\mu,W)+1)
  \int f (-\mathcal{L}_\mu^W) f \d\mu ,
 $$
 holds for any $f\in H^1(\mu,W)$.
 Next, because the injection of $H^1(\mu,W)$ in $L^2(\mu)$ is compact and that $H^1(\mu,W)$ is dense in $L^2(\mu)$, Theorem 7.3.2 in \cite{allaire2007numerical} ensures that $-\mathcal{L}_\mu^W$ admits the eigendecomposition \eqref{eq:EigendecompositionOfLW} and that $\lambda_i^W \rightarrow +\infty $ when $i\rightarrow +\infty$.
 By \eqref{eq:LW}, the first eigenpaire is trivial with $\lambda_1^W=0$ and $u_1^W=\mathds{1}$.
 Using the Courant--Fischer--Weyl min-max principle, we have
 $$
  \lambda_2^W = \inf_{\substack{f\in H^1(\mu,W) \\ \int f u_1^W \d\mu = 0 }} \frac{\int f (-\mathcal{L}_\mu^W)f \d\mu}{\int f^2\d\mu} ,
 $$
 Together with \eqref{eq:inversePoincareConstant}, this gives the identity in \eqref{eq:Lambda1}.

 Because $\lambda_i^W \rightarrow +\infty $, we have that the multiplicity $m$ of $\lambda_2^W$ is necessarily finite.
 As a consequence, the set of the minimizers of
 $$\frac{\int f (-\mathcal{L}_\mu^W)f \d\mu}{\int f^2\d\mu}
 = \int \trace\left( W \, \frac{  \nabla f \nabla f^\top }{\int f^2\d\mu} \right) \d\mu ,
 $$
 over $f\in H^1(\mu,W)$ with $\int f\d\mu=0$ is exactly $\text{span}\{u_2^W,\hdots,u_{m+1}^W\}$. Thus, Theorem 2.4.18 of \cite{zalinescu2002convex} establishes that \eqref{eq:superGradientOfLW} holds.
 Finally, \eqref{eq:GradientOfLW} is a direct consequence of \eqref{eq:superGradientOfLW}. This concludes the proof.
\end{proof}

Let us comment on the assumptions of Proposition \ref{prop:EigendecompositionOfLW}. First, the space $H^1(\mu,W)$ is dense in $L^2(\mu)$ is classically obtained by showing that the function space $\mathcal{C}^\infty_c(\R^d)$, which is of class $\mathcal{C}^\infty$ with compact support in $\R^d$, is dense in both $H^1(\mu,W)$ and $L^2(\mu)$. For classical (unweighted) Sobolev spaces, it is well known that $\mathcal{C}^\infty_c(\R^d)$ is dense in  $L^2(\R^d)$ and in $H^1(\R^d)$, see Theorem 4.12 and Corollary 9.8 of \cite{brezis2011functional} for details. For weighted Sobolev spaces, the result still holds under the assumption that the Lebesgue density of $\mu$ is in the Muckenhoupt's $\mathcal{A}_2$-class, see Theorem 1.1 of \cite{nakai2004density}. In particular, any measure $\mu$ with bounded support $\Omega=\text{supp}(\mu)$ and whose Lebesgue density is bounded away from 0 and $+\infty$ is in the Muckenhoupt's $\mathcal{A}_2$-class. If we further assume that $\alpha I_d\preceq W\preceq \beta I_d $ for some constant $0<\alpha\leq\beta<\infty$, then $H^1(\mu,W)=H^1(\mu)$ is dense in $L^2(\mu)$.

Second, the compact injection of $H^1(\mu,W)$ in $L^2(\mu)$ can be obtained via the Rellich–Kondrachov theorem, see Theorem 9.16 of \cite{brezis2011functional}. This theorem states that if $\Omega\subset\R^d$ is a bounded subset of class $\mathcal{C}^1$, then the injection of the (unweighted) Sobolev space $H^1(\Omega)$ in $L^2(\Omega)$ is compact. Therefore, if $\Omega=\text{supp}(\mu)$ and there exist some constant $0<\alpha\leq\beta<\infty$ such that $\alpha I_d\preceq W\preceq \beta I_d $ and $ \alpha\leq \frac{\d\mu}{\d x}\leq\beta$, then the space  $H^1(\Omega) = H^1(\mu,W)$ is dense in $L^2(\Omega)=L^2(\mu)$.

\section{Solution algorithm}\label{sec:NumericalSolution}

To account for the constraints $W(x)\in\mathcal{S}_+^d$ and $\int\trace(W)\d\mu=\trace(\Cov_\mu)$ in Problem \eqref{eq:minC}, we propose to use the following parametrization of $W$
\begin{equation}\label{eq:WofV}
 W(x) = V(x)^2 \frac{\trace(\Cov_\mu)}{\int \|V\|_F^2\d\mu},
\end{equation}
where $V:\R^d\rightarrow\mathcal{S}^d$ is a matrix field to be determined. Here, $\mathcal{S}^d\subset\R^{d\times d}$ denotes the set of symmetric matrices and $\|\cdot\|_F$ denotes the Frobenius norm such that $\|A\|_F^2=\trace(A^\top A)$.
Because any $W:\R^d\rightarrow\mathcal{S}_+^d$ with $\int\trace(W)\d\mu=\trace(\Cov_\mu)$ can be written as in \eqref{eq:WofV} for some $V$, we have that any maximizer of
\begin{align}
J(V) = \inf_{\substack{f\in H^1(\mu,V^2) \\ \int f\d\mu=0}}\frac{\int\nabla f^\top V^2 \nabla f\d\mu}{(\int\|V\|_F^2\d\mu)(\int f^2\d\mu)}
 \overset{ \eqref{eq:inversePoincareConstant} \&  \eqref{eq:WofV} }{=} \frac{1}{\trace(\Cov_\mu) C(\mu,W)} ,
\end{align}
yields a solution $W$ to \eqref{eq:minC}.
By parametrizing $W$ with $V$, the concavity of the constrained problem \eqref{eq:minC} (see Proposition \ref{prop:CisConvex}) no longer holds. Instead, we obtain an \emph{unconstrained} maximization problem
\begin{equation}\label{eq:maxJ}
\max_{V:\R^d\rightarrow\mathcal{S}^d} J(V) ,
\end{equation}
which can be much easier to solve in practice.

\subsection{Gradient ascent algorithms}\label{sec:Gradient}

Proposition \ref{prop:EigendecompositionOfLW} permits one to write
\[
J(V)=\frac{\lambda_2^{V^2}}{\int \|V\|_F^2\d\mu},
\]
where $\lambda_2^{V^2}$ denotes the first nonzero eigenvalue of the operator $\mathcal{L}_\mu^{V^2}$ as in \eqref{eq:LW}.
If we assume that $\lambda_2^{V^2}$ has multiplicity one, then $W\mapsto\lambda_2^{W}$ is differentiable around $W=V^2$. In this case, $J$ is also differentiable around $V$ as the product and composition of differentiable functions, and its gradient is given by
\begin{equation}\label{eq:GradientOfJ}
 \nabla J(V)
 = \frac{ GV+VG-2J(V)V}{\int\|V\|^2_F\d\mu},
 \quad\text{with } G = \frac{(\nabla u_2^{V^2} )( \nabla u_2^{V^2} )^\top }{ \int (u_2^{V^2})^2 \d\mu } .
\end{equation}
If the multiplicity of $\lambda_2^{V^2}$ is larger than one, then $J$ is no longer differentiable and does not admit a super-differential (because $J$ is not concave).
Fortunately, $J$ still possesses a \emph{generalized gradient} in the sense of \cite[Chapter 10]{clarke2013functional} which is given with a similar expression at in \eqref{eq:GradientOfJ}, albeit $G$ now belonging to the super-differential $\partial_W(\lambda_2^{V^2})$. This is, however, not further analyzed in the present work.

The gradient ascent algorithm with constant step size $\rho\geq0$ consists of a sequence of updates $\{V^{(k)}\}_{k\geq0}$ defined as
\begin{align}
 V^{(k+1/2)} &= V^{(k)}+\rho \nabla J(V^{(k)}) , \label{eq:GradientAscentStep}\\
 V^{(k+1)} &= \frac{V^{(k+1/2)}}{(\int\|V^{(k+1/2)}\|_F^2\d\mu)^{1/2}}.
\end{align}
Given the invariance condition $J(\alpha V) = J(V)$ for all $\alpha\neq0$, the last step in the above normalizes the iterative solution without modifying the value of the objective function, i.e., $J(V^{(k+1)}) = J(V^{(k+1/2)})$.
While gradient ascent is universally popular, alternative methods such as the momentum method \cite{qian1999momentum} and Nesterov's accelerated gradient algorithm \cite{nesterov1983method} can result in significantly faster convergence to the optimum.
In our case, the momentum method is defined by the following iterations
\begin{align}
 D^{(k+1)} &= \alpha D^{(k)} + \rho \nabla J(V^{(k)}) , \label{eq:MOM1}\\
 V^{(k+1/2)} &= V^{(k)} + D^{(k+1)} , \label{eq:MOM2}\\
 V^{(k+1)} &= \frac{V^{(k+1/2)}}{(\int\|V^{(k+1/2)}\|_F^2\d\mu)^{1/2}} \label{eq:MOM3} .
\end{align}
It consists of continuing updating along the previous update direction $D^{(k)}$, scaled by a factor $0\leq \alpha<1$. Nesterov's accelerated gradient algorithm replaces the direction in \eqref{eq:MOM1} with
\begin{equation}\label{eq:NAG}
 D^{(k+1)} = \alpha_k D^{(k)} + \rho \nabla J(V^{(k)} + \alpha_k D^{(k)} ) ,
\end{equation}
where $0\leq \alpha_k<1$ depends on the iteration number, see \cite{sutskever2013importance}. Nesterov proposed to use $\alpha_k = 1-3/(5+k)$, which is adopted in our experiments.

We compute the gradient $\nabla J(V^{(k)})$ as in \eqref{eq:GradientOfJ}, which assumes $\lambda_2^{(V^{(k)})^2}$ has multiplicity one. As shown in the numerical results, the first two nonzero eigenvalues cross each other several times during the iterative update and they tend to be equal at the optimum.
Even though the first two nonzero eigenvalues are always numerically different---which indicates that the objective function $J$ may possess several discontinuities in the gradient---the use of either the momentum method or Nesterov's accelerated gradient turns out to be very efficient for solving \eqref{eq:maxJ} when the iterates approach the optimum.

\subsection{Finite element discretization}

To numerically compute the gradient $\nabla J(V^{(k)})$ in \eqref{eq:GradientOfJ}, we employ the finite element method \cite{allaire2007numerical,zienkiewicz2000finite}.
We consider a triangulation $\mathcal{T}_M$ of the support of $\mu$ made with $M$ simplices  (\emph{i.e.} triangles when $d=2$) and $N$ nodes such that $\text{supp}(\mu)=\omega_1\cup\hdots\cup \omega_M$.
Here, we assume that the support of $\mu$ is bounded.
We define the finite element space $\mathcal{V}_N\subset H^1(\mu,W)$ as
$$
 \mathcal{V}_N = \text{span}\{\phi_1,\hdots,\phi_N\} ,
 \quad \phi_i:\text{continuous and piecewise affine},
$$
such that $\phi_i(x_j)=\delta_{i,j}$, where $x_j$ is the $j$-th node of $\mathcal{T}_M$.
The function
$$
 J_N(V) = \inf_{\substack{f\in \mathcal{V}_N \\ \int f\d\mu=0}}\frac{\int\nabla f^\top V^2 \nabla f\d\mu}{(\int\|V\|_F^2\d\mu)(\int f^2\d\mu)} ,
$$
yields an approximation to $J(V)$ such that $J_N(V) \rightarrow J(V)$ when $N\rightarrow \infty$, see \emph{e.g.} Theorem 7.4.4 from \cite{allaire2007numerical}.
We propose to replace $J$ by $J_N$ in the gradient algorithms \eqref{eq:GradientAscentStep}, \eqref{eq:MOM1} and \eqref{eq:NAG}. Next, we show how to compute $\nabla J_N(V)$.

For a given matrix field $W$ as in \eqref{eq:WofV}, we let $K^W\in\R^{N\times N} $ and $ M\in\R^{N\times N}$ be the matrices defined by
\begin{align}
 K_{i,j}^{W} = \int \nabla\phi_i^\top  W \nabla\phi_j \d\mu,
 \qquad\text{and}\qquad
 M_{i,j} = \int \phi_i\phi_j \d\mu,
 \label{eq:defKM}
\end{align}
for all $1\leq i,j\leq N$. Consider the generalized eigendecomposition of the matrix pair $(K^W,M)$
$$
 K^W U_i^W = \lambda^W_{N,i} M U_i^W ,
$$
where $(U_j^W)^\top  M U_i^W = \delta_{i,j}$ for any $1\leq i,j\leq N$ and where the eigenvalues $\lambda^W_{N,1},\hdots,\lambda^W_{N,N}$ are listed in the increasing order.
Because the finite element space $\mathcal{V}_N$ contains the constant function, the first eigenvalue is trivial $\lambda^W_{N,1}=0$ and the corresponding eigenvector $U_1^W =U_1= (1,\hdots,1)^\top  $ represents the constant function $\mathds{1}=\sum_{i=1}^N U_{1,i}\phi_i$ onto the finite element basis.
Also, because $\int f \d\mu=U^\top  M U_1$ for any $f=\sum_{i=1}^N U_i\phi_i$, we can express $J_N(V)$ as follow
$$
 J_N(V)
 = \frac{1}{\int\|V\|_2^2\d\mu}\inf_{\substack{U\in \R^N \\ U^\top  M U_1 = 0 }} \frac{U^\top K^{V^2} U}{U^\top MU}
 = \frac{\lambda_{N,2}^{V^2}}{\int\|V\|_2^2\d\mu} .
$$
As in Equation \eqref{eq:GradientOfJ}, if the multiplicity of $\lambda_{N,2}^{V^2}$ is one, then $J_N$ admits a gradient at $V$ which is given by
\begin{equation} \label{eq:GradientOfJN}
\begin{split}
 \nabla J_N(V) &= \frac{ G_N V+VG_N -2J_N(V) V}{\int\|V\|^2_F\d\mu}, \\
 &\text{with } G_N = \frac{ \left(\sum_{i=1}^N (U_2^{V^2})_i \nabla \phi_i \right)\left(\sum_{i=1}^N (U_2^{V^2})_i \nabla \phi_i \right)^\top}{ (U_2^{V^2})^\top M (U_2^{V^2}) }.
\end{split}
\end{equation}
Thus, to compute $\nabla J_N(V)$, one needs first to assemble the finite element matrices $K^{V^2}$ and $M$ and to compute the second smallest generalized eigenpair $(\lambda_2^{V^2},U_2^{V^2})$ of $(K^{V^2},M)$. Then, $\nabla J_N(V)$ is assembled according to \eqref{eq:GradientOfJN}.
Notice that, because $\phi_i$ are piecewise affine, the matrix field $G_N$ is piecewise constant on each element $\omega_i$ of the mesh.
In addition, if the field $V$ is also piecewise constant of the form of
\begin{equation}
 V(x) =  \sum_{m=1}^M V_m \mathds{1}_{\omega_m}(x) , \text{ where } V_m\in\R^{d\times d},
\end{equation}
then $\nabla J_N(V)$ as in \eqref{eq:GradientOfJN} is also piecewise constant.
Here, $\mathds{1}_{\omega_m}$ is the indicator function of the $m$-th element $\omega_m$.
Therefore, any of the iterates of the gradient algorithms in Equations \eqref{eq:GradientAscentStep}--\eqref{eq:NAG} remains piecewise constant, and so is the resulting symmetric positive field $W$ given by \eqref{eq:WofV}.
The resulting algorithm is summarized in Algorithm \ref{algo}.

\begin{algorithm}[!htb]
\label{alg}
   \caption{Optimal metric using finite element method and momentum algorithm}
\begin{algorithmic}[1]
   \STATE {\bfseries Input}: measure $\mu$, finite element partition $\text{supp}(\mu)=\omega_1\cup\hdots\cup\omega_M$, step size $\rho>0$ and momentum parameter $0\leq\alpha<1$.
   \STATE Define the finite element space $\mathcal{V}_N=\text{span}\{\phi_1,\hdots,\phi_N\}$ of piecewise affine functions.
   \STATE For all $i,j\leq N$ and $m\leq M$, precompute
   $$
   \mu_m = \int_{\omega_m}\d\mu
   \quad\text{and}\quad
   M_{i,j} = \int \phi_i\phi_j \d\mu
   \quad\text{and}\quad
   K_{i,j,m} = \int_{\omega_m} \nabla\phi_i\nabla\phi_j^\top \d\mu \in\R^{d\times d}.
   $$
   \STATE For all $m\leq M$, initialize $V^{(0)}_m=I_d/d$ and $D^{(0)}_m=0$.
   \STATE Set $k=0$.
   \WHILE{not converged}
   \STATE For all $m\leq M$, compute the metric
        \hfill (see Equation \eqref{eq:WofV})
   $$
    W^{(k)}_m = (V^{(k)}_m)^2 \frac{\trace(\Cov_\mu)}{ \sum_{m=1}^M \|V^{(k)}_m\|_F^2  \mu_m }
   $$
   \STATE Assemble the matrix $K^{(k)} \in\R^{N\times N}$ given by
   \hfill (see Equation \eqref{eq:defKM})
   $$
    K_{i,j}^{(k)}
    =
    \sum_{i=1}^M  \trace\left(K_{i,j,m} W^{(k)}_m \right)
   $$
   \STATE Compute the smallest non-zero generalized eigenpair $(\lambda^{(k)},U^{(k)})\in \R_{>0}\times\R^N$ solution to
   $$K^{(k)} U^{(k)} = \lambda^{(k)} M U^{(k)} $$
   \STATE For all $m\leq M$, compute
        \hfill (see Equation \eqref{eq:GradientOfJN})
$$
 G_m^{(k)} = \frac{\sum_{i,j=1}^N U_i^{(k)} U_j^{(k)} K_{i,j,m}}{\sum_{i,j=1}^N U_i^{(k)} U_j^{(k)} M_{i,j}}
$$
and then $\nabla J_m^{(k)} = G_m^{(k)} V_m^{(k)} + V_m^{(k)} G_m^{(k)} - 2 \lambda^{(k)} V_m^{(k)}$
   \STATE Compute the increment
   $
    D_m^{(k+1)} = \alpha D_m^{(k)} + \rho \nabla J_m^{(k)}
   $
   \hfill (see Equation \eqref{eq:MOM1})
   \STATE Update $V_m^{(k+1/2)} = V_m^{(k)} + D_m^{(k+1)}$ \hfill (see Equation \eqref{eq:MOM2})
    \STATE Normalize $V_m^{(k+1)} = V_m^{(k+1/2)} / (\sum_{m'=1}^M \|V^{(k)}_{m'}\|_F^2  \mu_{m'} )^{1/2} $
        \hfill (see Equation \eqref{eq:MOM3})
   \STATE $k\leftarrow k+1$
   \ENDWHILE
   \STATE {\bfseries Output}: Metric $W^{(k)} = \sum_{m=1}^M W^{(k)}_m \mathds{1}_{\omega_m}$ and the associated Poincaré constant $C(\mu,W^{(k)}) = 1/\lambda^{(k)}$.
\end{algorithmic}
\label{algo}
\end{algorithm}

\section{Numerical experiments}\label{sec:Applications}

We demonstrate the proposed method on four benchmarks, each of them in dimension $d=2$.
Code for reproducing all of these numerical results is available online\footnote{\url{https://gitlab.inria.fr/ozahm/optimal-riemannian-metric-for-poincare-inequalities.git}}. The first is a Gaussian mixture of the form of
$$
 \d\mu_1(x) \propto \sum_{i=1}^3 \exp\left( -\frac{\|x-x_i\|^2}{\sigma^2}\right) \d x,
$$
where $\sigma^2 = 0.025$, $x_1=(0,0.5)$, $x_2=(\sqrt{3}/4,\sqrt{3}/4)$ and $x_3=(-\sqrt{3}/4,\sqrt{3}/4)$.
The second distribution is concentrated on a ring of radius $r$ and is defined by
$$
 \d\mu_2(x) \propto \exp\left( -\frac{(\|x\|-r)^2}{\sigma^2}\right)\d x,
$$
where $\sigma^2 = 0.0032$ and $r=0.65$.
The third is a uniform measure on a $H$-shaped domain $\Omega\subset\R^2$, meaning
$$
 \d\mu_3(x) \propto \mathds{1}_{\Omega}(x) \d x.
$$
The fourth is a slight modification of $\mu_3$ which consists in enlarging the support of $\mu_3$ to the convex hull $\text{conv}(\Omega)$ of the $H$-shaped $\Omega$ as follow
$$
 \d\mu_4(x) \propto \d\mu_3(x) + \varepsilon\, \mathds{1}_{\text{conv}(\Omega)}(x) \d x,
$$
where $\varepsilon=10^{-7}$. The goal is to create a measure which is close to $\mu_3$ while being supported on a convex domain.
That way, $\mu_4$ satisfies the (sufficient) conditions for being a moment measure with a smooth enough convex potential $\varphi$ so that Theorem \ref{th:WstartMomentMap} applies, see \cite{berman2013real}. This will be illustrated later in Figure \ref{fig:optimalW}.

The measures $\mu_1,\hdots,\mu_4$ as well as the finite element discretizations of their supports, are represented in Figure \ref{fig:discretization}. The Poincaré constant of these measures is respectively
$$
 C(\mu_1) = 61.79
 ,\quad
 C(\mu_2) = 1.971
 ,\quad
 C(\mu_3) = 15.01
 ,\quad
 C(\mu_4) = 15.01.
$$

\begin{figure}
\centering
 \begin{subfigure}{0.24\textwidth}
        \centering
        \includegraphics[width=\textwidth]{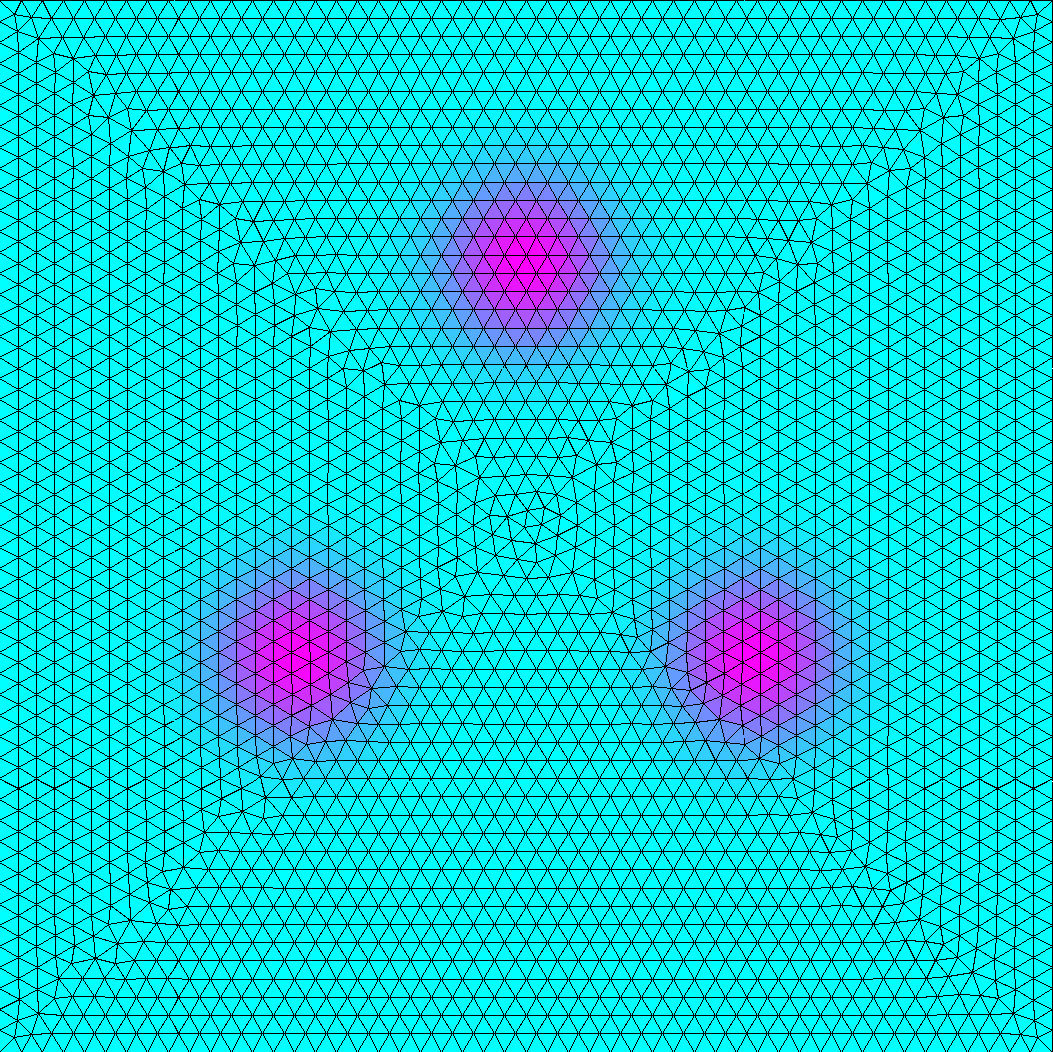}
        \caption{Tri-modal $\mu_1$}
    \end{subfigure}
    \begin{subfigure}{0.24\textwidth}
        \centering
        \includegraphics[width=\textwidth]{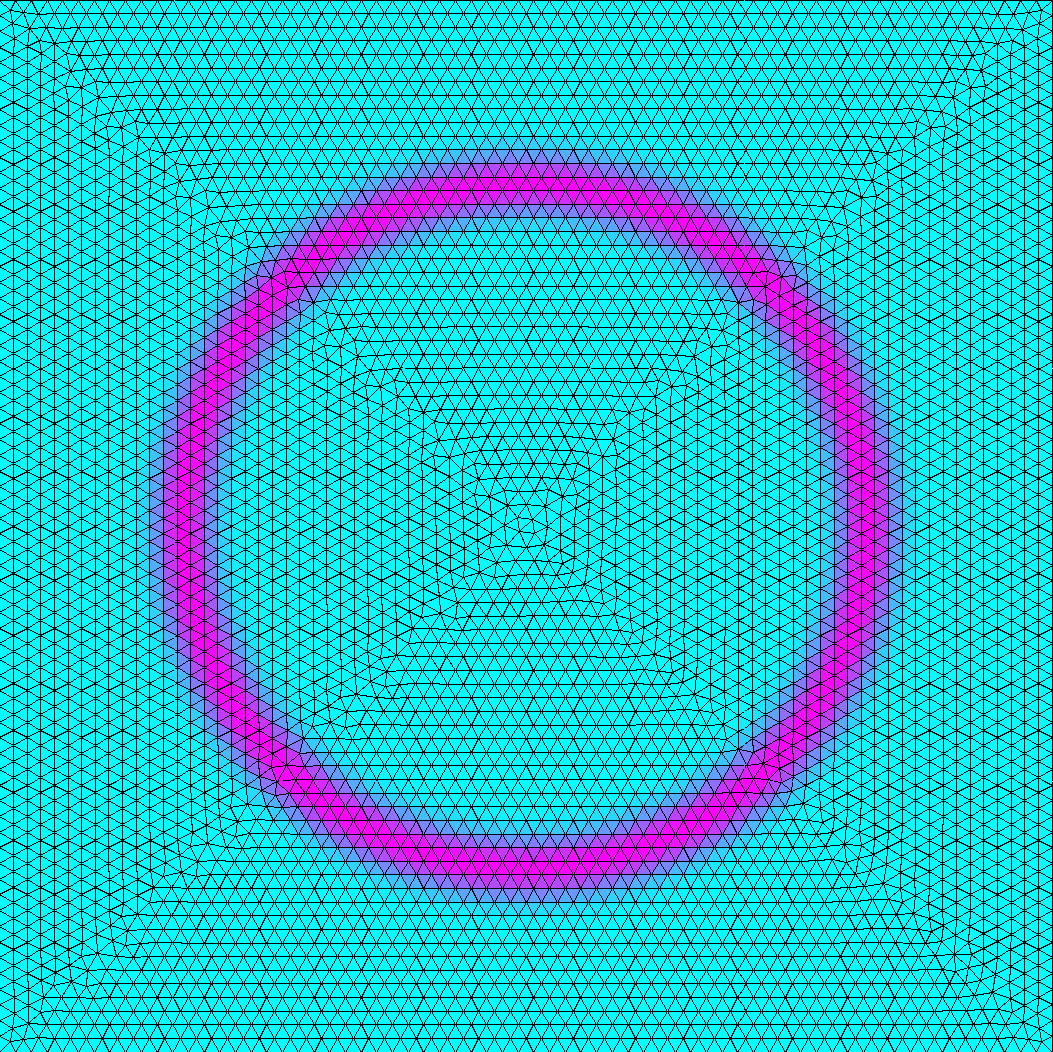}
        \caption{Ring $\mu_2$}
    \end{subfigure}
    \begin{subfigure}{0.24\textwidth}
        \centering
        \includegraphics[width=\textwidth]{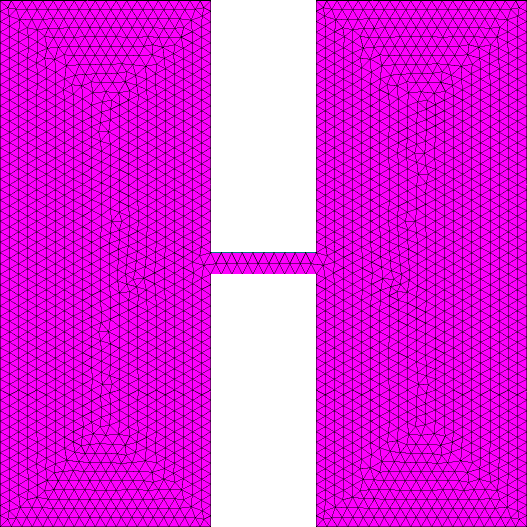}
        \caption{H-shaped $\mu_3$}
    \end{subfigure}
    \begin{subfigure}{0.24\textwidth}
        \centering
        \includegraphics[width=\textwidth]{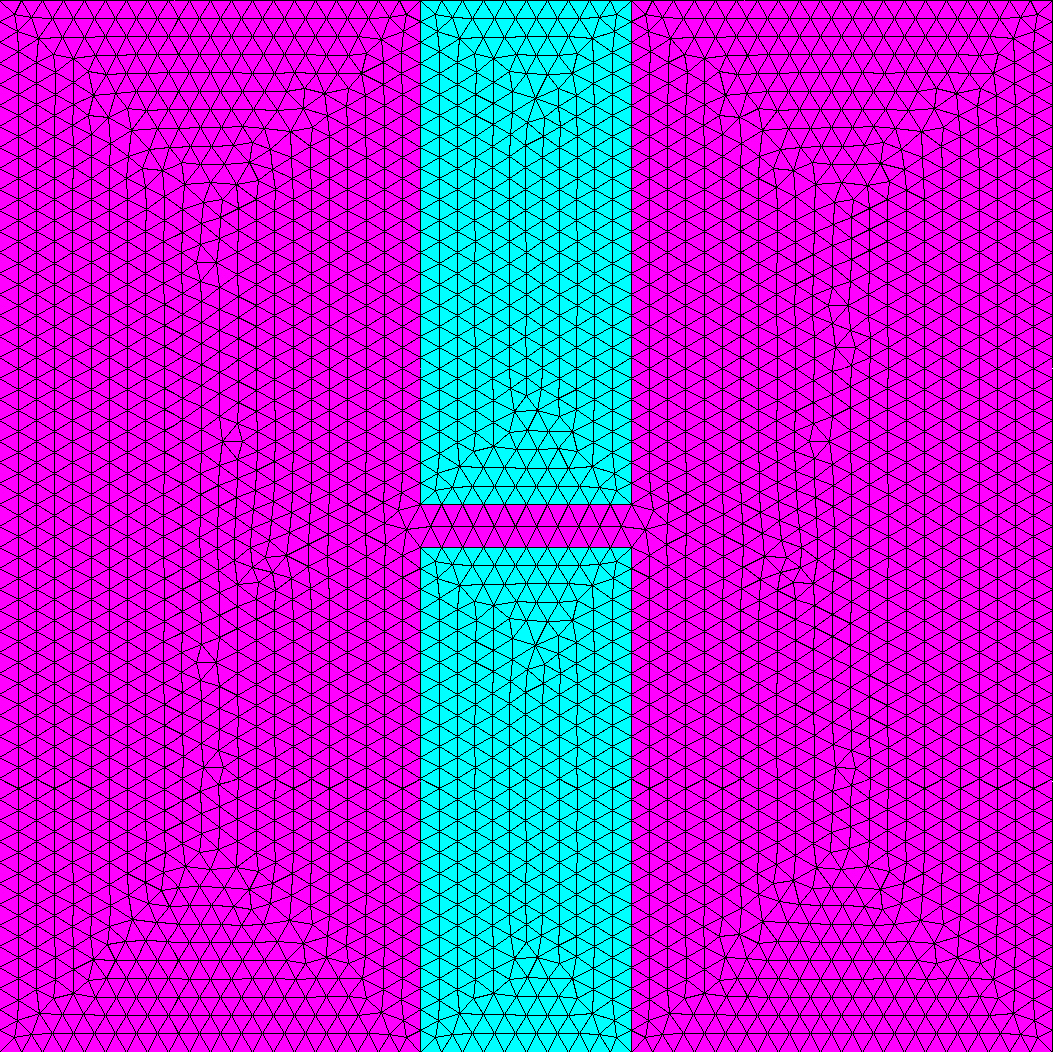}
        \caption{H-shaped $\mu_4$}
    \end{subfigure}
 \caption{The four benchmark measures and the finite element discretization of their support (cyan indicates low probability density and purple indicates high probability density).
 The number of elements in the mesh are respectively $M_1=5718$, $M_2=10256$, $M_3=4610$ and $M_4=5742$.
 }
    \label{fig:discretization}
\end{figure}

\subsection{Gradient ascent methods on the tri-modal measure}

Figure \ref{fig:convergence_trimodal} reports the evolution of the spectrum of $\mathcal{L}_{\mu_1}^{W}$ during the optimization process.
We observe that all the gradient-ascent methods we considered yield $\lambda_1\geq0.9$ (and therefore $C(\mu,W)\leq 1.12$) after 15 iterations.
The oscillations between the two first nonzero eigenvalues $\lambda_2$ and $\lambda_3$ indicate that those eigenvalues are constantly crossing each other during the iterations.
When approaching optimum, $\lambda_2,\lambda_3\rightarrow1$, Nesterov's acceleration yields more stable behaviour in the iterates compared to the other methods. Running Nesterov's acceleration for $k=100$ iterations, we obtain $\lambda_2^{(k)}=0.9998$, showing that one can bring the Poincaré constant $C(\mu,W^{(k)})$ arbitrarily close to $1$.

\begin{figure}
\centering
 \begin{subfigure}{0.32\textwidth}
        \centering
        \includegraphics[width=\textwidth]{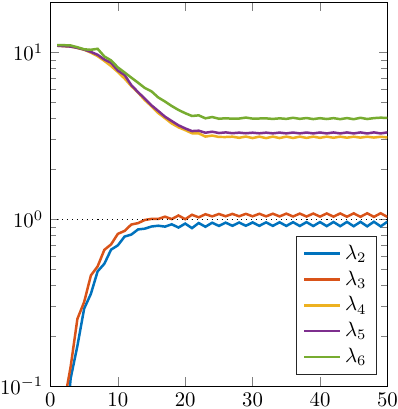}
        \caption{Gradient ascent}
    \end{subfigure}
    \begin{subfigure}{0.32\textwidth}
        \centering
        \includegraphics[width=\textwidth]{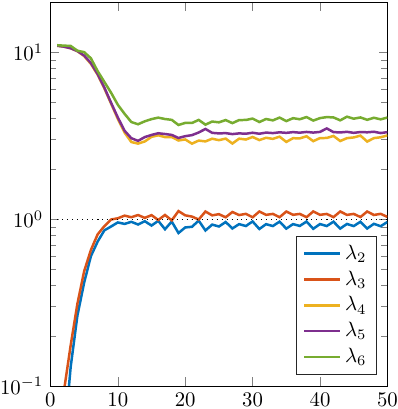}
        \caption{Momentum}
    \end{subfigure}
    \begin{subfigure}{0.32\textwidth}
        \centering
        \includegraphics[width=\textwidth]{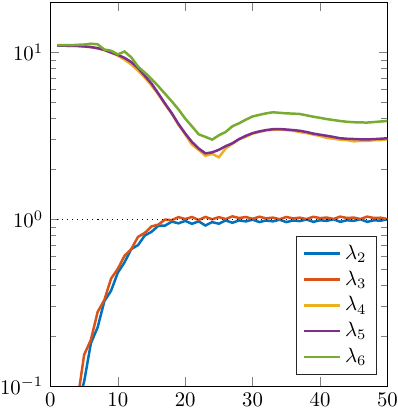}
        \caption{Nesterov}
    \end{subfigure}
 \caption{(Tri-modal $\mu_1$) Evolution of the five first nonzero eigenvalues of the diffusion operator $\mathcal{L}_\mu^{W}$ during the iterative process of the gradient-ascent methods, with $\rho=0.01$ and $\alpha=0$ (Gradient ascent), $\alpha=0.5$ (Momentum) and $\alpha_k = 1-3/(5+k)$ (Nesterov).}
    \label{fig:convergence_trimodal}
\end{figure}

Figure \ref{fig:modes_trimodal} shows the first 5 nontrivial eigenvectors of $\mathcal{L}_{\mu_1}^{W}$ for the (initial) constant metric $W^{(0)}=I_d  \trace(\Cov_\mu)/d$ and the (final) metric $W^{(k)}$ obtained using Nesterov's acceleration with $k=100$ iterations. We observe that the first two eigenvectors are close to affine functions, which confirms the result of Theorem \ref{th:OptimalityCondition}, reflecting the optimality of $W^{(k)}$.

\begin{figure}
\centering

 \begin{subfigure}{0.19\textwidth}
        \centering
        \includegraphics[width=\textwidth]{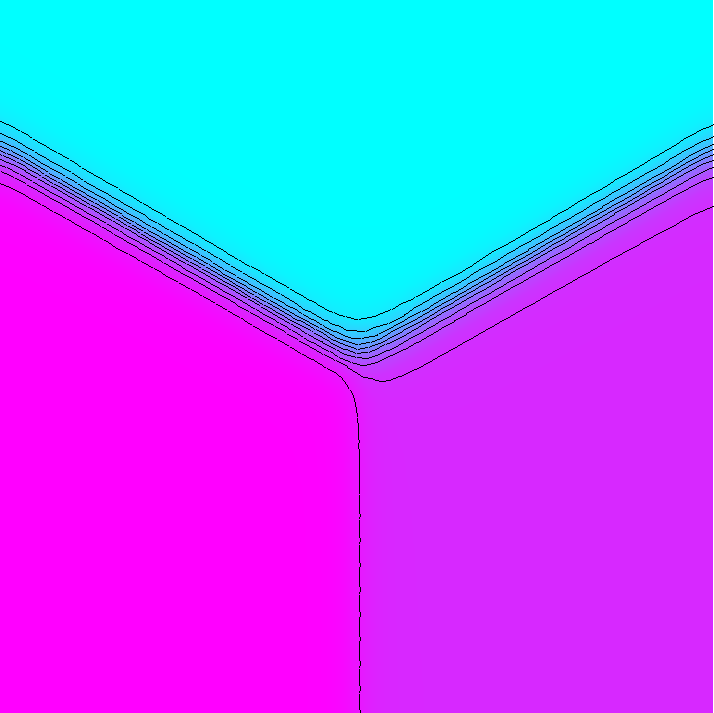}
        \caption{$\lambda_2=0.0162$}
    \end{subfigure}
    \begin{subfigure}{0.19\textwidth}
        \centering
        \includegraphics[width=\textwidth]{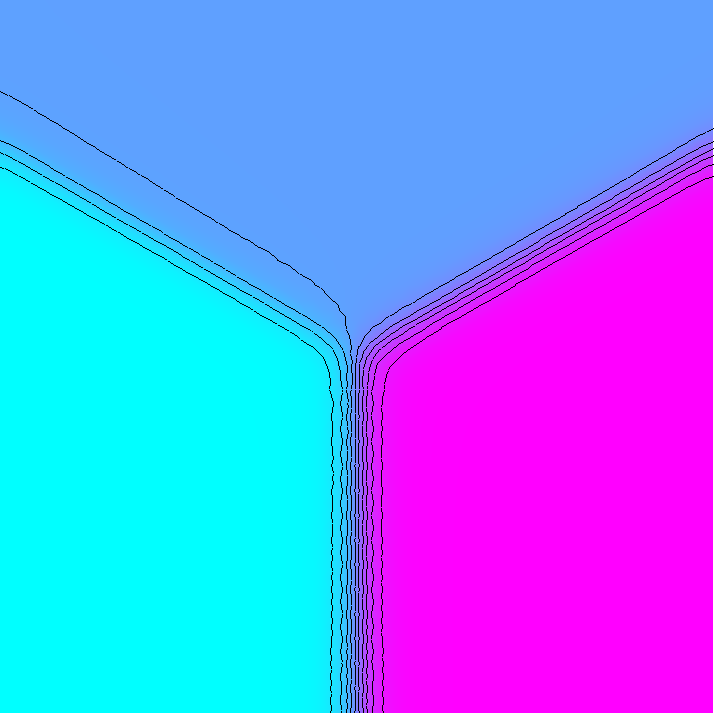}
        \caption{$\lambda_3=0.0162$}
    \end{subfigure}
    \begin{subfigure}{0.19\textwidth}
        \centering
        \includegraphics[width=\textwidth]{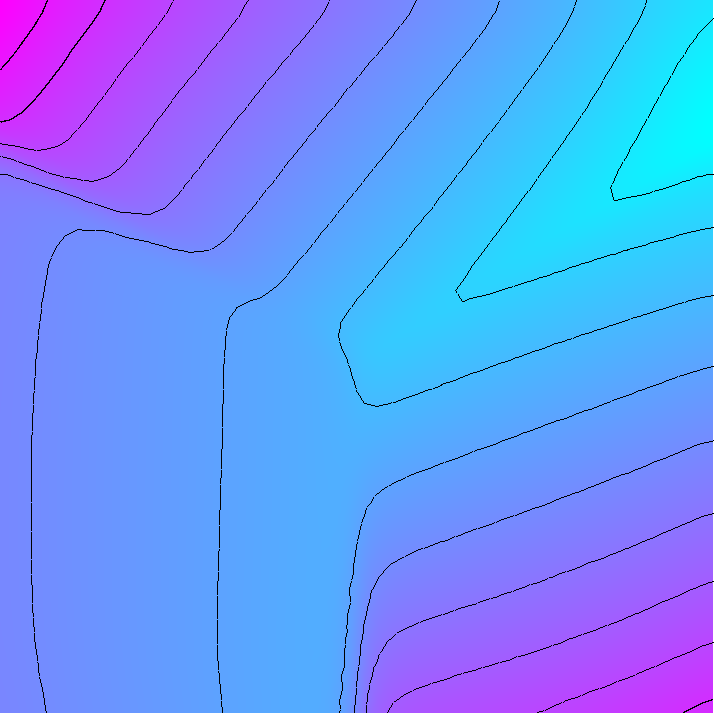}
        \caption{$\lambda_4=10.9513$}
    \end{subfigure}
    \begin{subfigure}{0.19\textwidth}
        \centering
        \includegraphics[width=\textwidth]{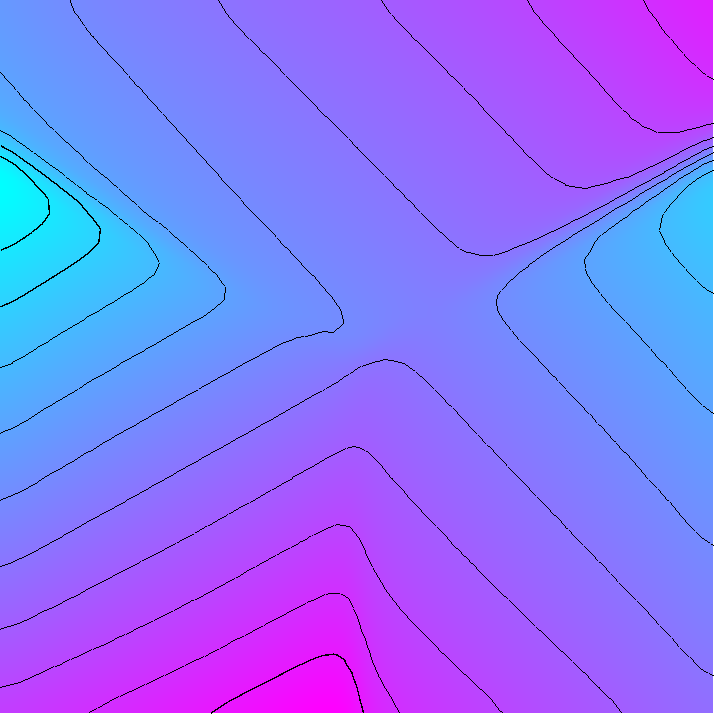}
        \caption{$\lambda_5=10.9517$}
    \end{subfigure}
    \begin{subfigure}{0.19\textwidth}
        \centering
        \includegraphics[width=\textwidth]{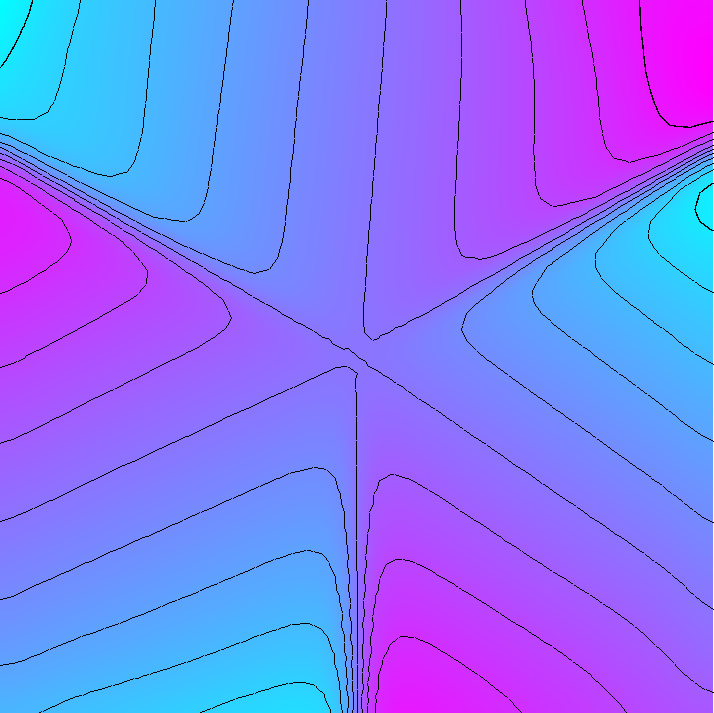}
        \caption{$\lambda_6=11.0290$}
    \end{subfigure}
 \begin{subfigure}{0.19\textwidth}
        \centering
        \includegraphics[width=\textwidth]{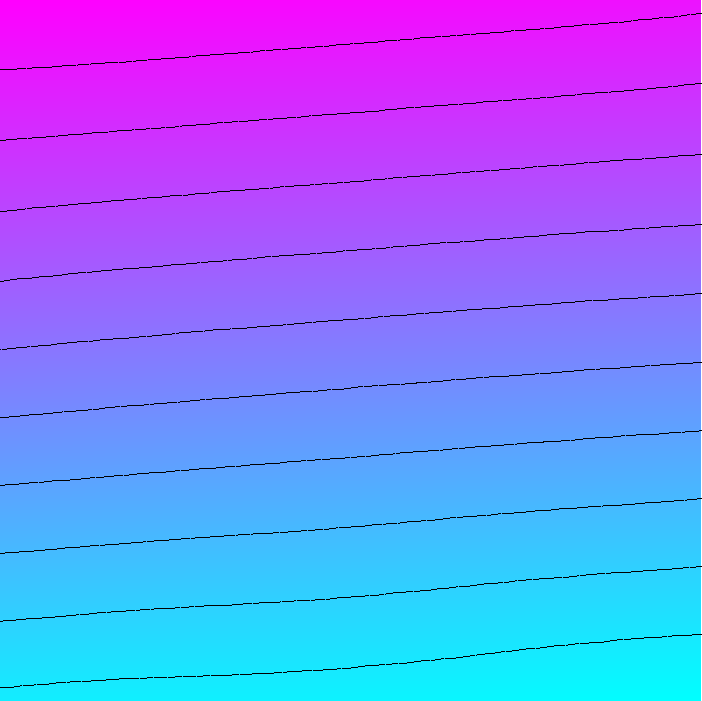}
        \caption{$\lambda_2=0.9998$}
    \end{subfigure}
    \begin{subfigure}{0.19\textwidth}
        \centering
        \includegraphics[width=\textwidth]{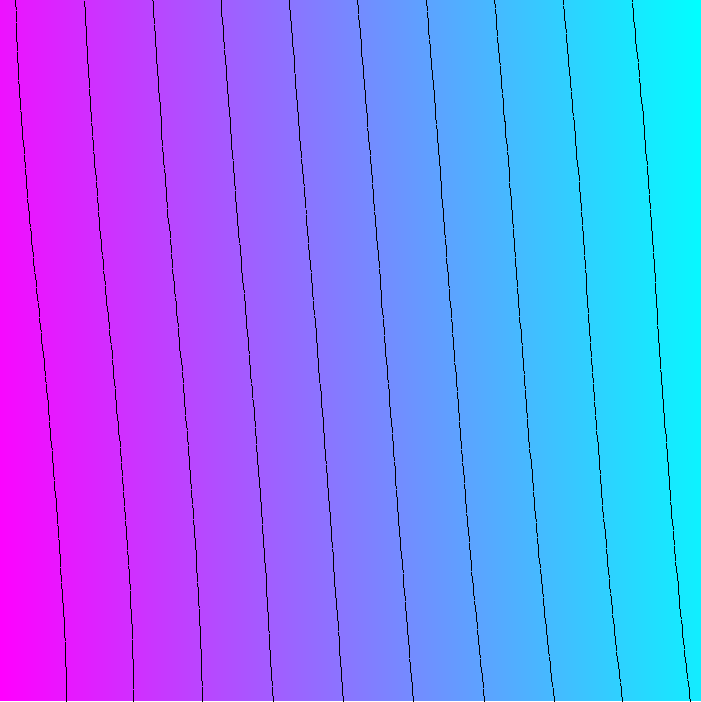}
        \caption{$\lambda_3=1.0002$}
    \end{subfigure}
    \begin{subfigure}{0.19\textwidth}
        \centering
        \includegraphics[width=\textwidth]{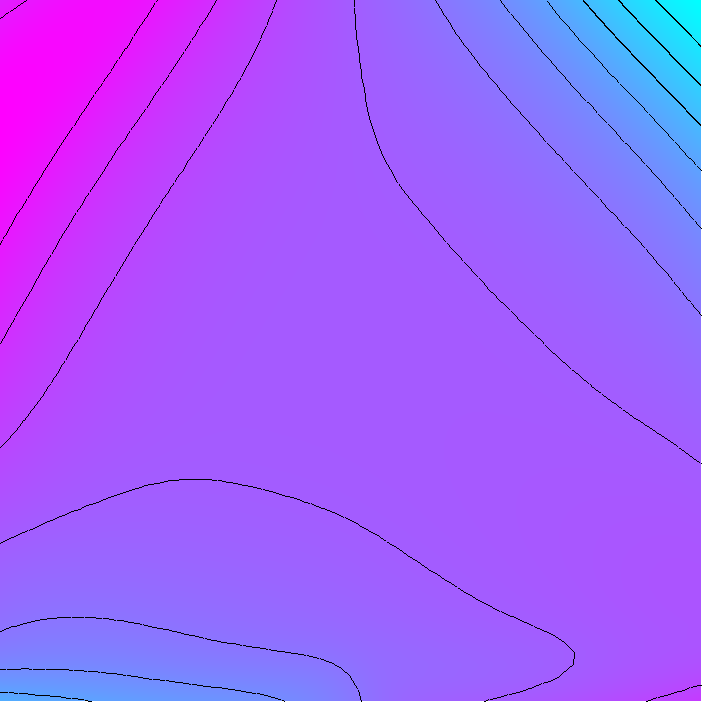}
        \caption{$\lambda_4=3.0973$}
    \end{subfigure}
    \begin{subfigure}{0.19\textwidth}
        \centering
        \includegraphics[width=\textwidth]{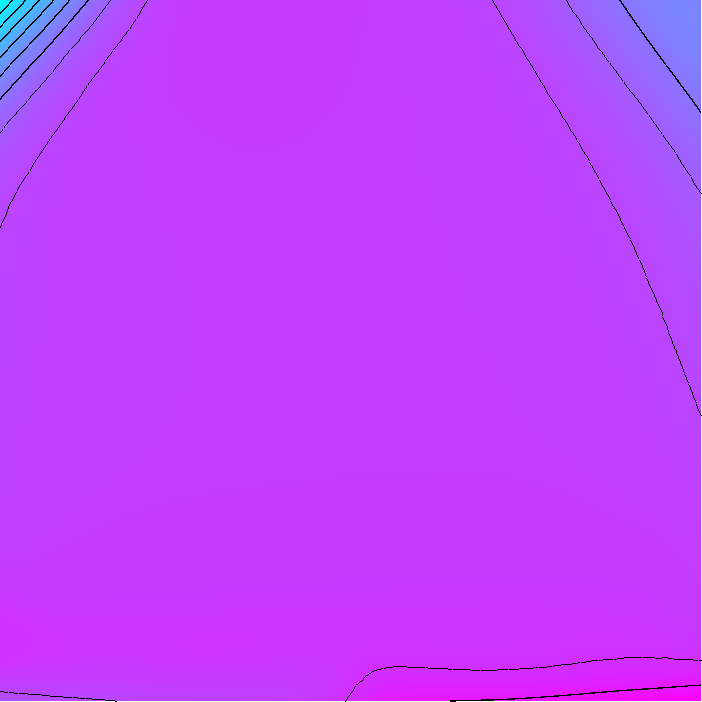}
        \caption{$\lambda_5=3.1311$}
    \end{subfigure}
    \begin{subfigure}{0.19\textwidth}
        \centering
        \includegraphics[width=\textwidth]{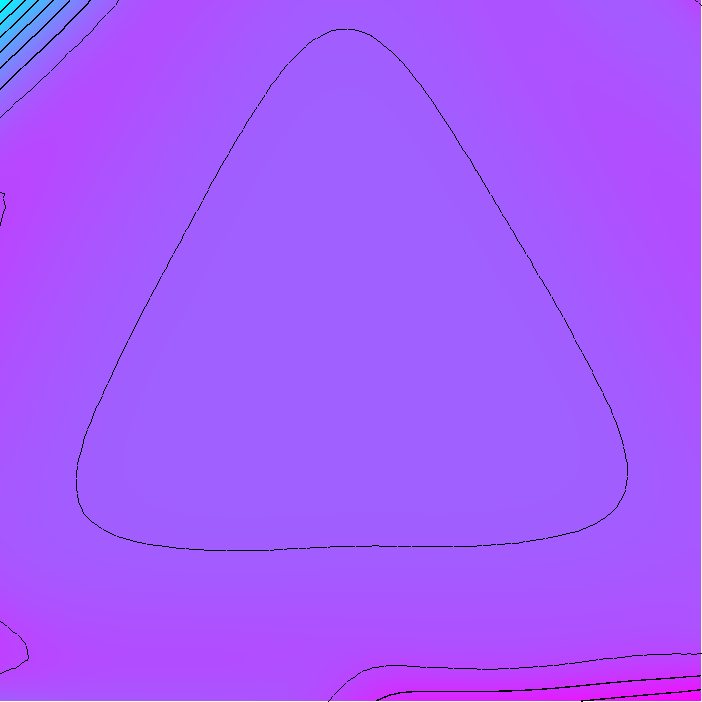}
        \caption{$\lambda_6=3.8463$}
    \end{subfigure}
 \caption{(Tri-modal $\mu_1$) Eigenvectors $u_2^W,\hdots,u_6^W$ of the diffusion operator $\mathcal{L}_{\mu_1}^{W}$ for  the constant metric $W^{(0)}=I_d  \frac{\trace(\Cov_\mu)}{d}$ (top row) and the computed metric $W^{(k)}$ obtained using $k=100$ iterations of Nesterov's acceleration.}
    \label{fig:modes_trimodal}
\end{figure}

\subsection{Optimal metrics}

Figure \ref{fig:optimalW} shows the optimal metric $W$ on the four test cases. To visualize the field of $2\times2$ positive-definite matrices, we plot the trace of $W$ (in log-scale for improved visibility) and the ellipses whose main directions are the eigenvectors of $W(x)$ to indicate the anisotropy of the field $W$.

For the tri-modal measure $\mu_1$, we see in Figure \ref{fig:optimalW_1} that $W$ is low and isotropic around the three modes, and signiﬁcantly high and anisotropic across the modes.
For the ring measure $\mu_2$, we observe in Figure \ref{fig:optimalW_2} that $W$ admits a singularity at the origin ($\trace(W) \approx 10^{10}$) and is anisotropic in regions where $\mu_2$ is high (the ring of radius $r$).

For the $H$-shaped measure $\mu_3$, Figure \ref{fig:optimalW_3} indicates a high value of $W$ in the ``bridge'' connecting the left and right parts of the domain.
It is important to notice that, for this benchmark, the resulting (reciprocal of the) Poincaré constant is $C(\mu_3,W)^{-1}=0.9166$ which is much higher than $1$ compared with the other test cases.
A credible explanation for this is that $\mu_3$ is not a moment measure, and so Theorem \ref{th:WstartMomentMap} does not apply.
In contrast, the measure $\mu_4\approx\mu_3$ has a convex support, and our algorithm yields a Poincaré constant with reciprocal $C(\mu_4,W^{(k)})^{-1}=0.9989$, which can be made arbitrarily close to $1$ with adding more number of iterations. We notice that in the region with low-probability density ($\d\mu_4\approx \varepsilon \d x$), the metric $W$ degenerates ($\trace(W)\approx \varepsilon^{-1}$) and is highly anisotropic.

\begin{figure}
\centering
 \begin{subfigure}{0.48\textwidth}
        \centering
        \includegraphics[width=\textwidth]{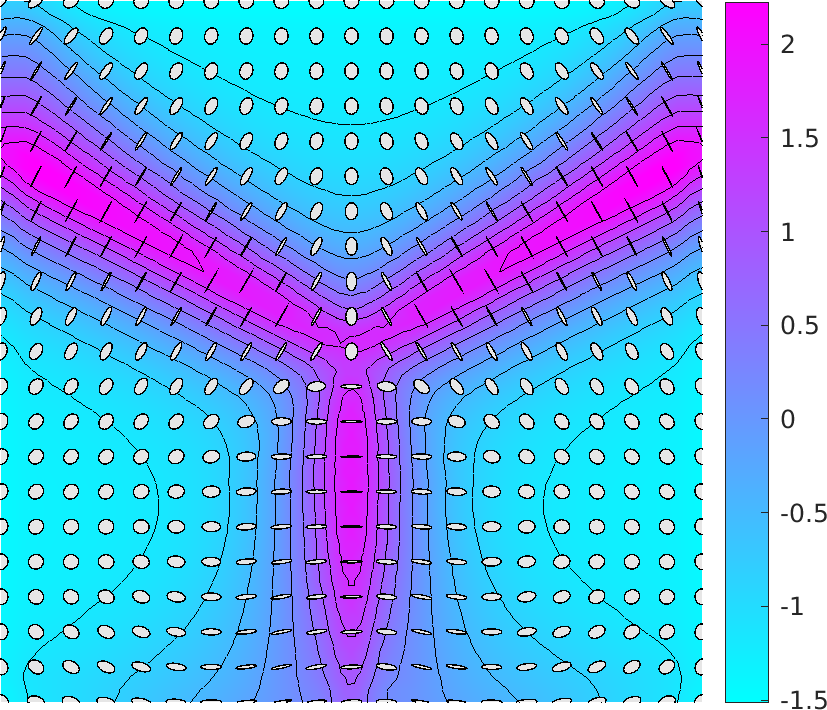}
        \caption{Tri-modal $\mu_1$, $C(\mu_1,W^{(k)})^{-1}=0.9998$}
        \label{fig:optimalW_1}
    \end{subfigure}
    \begin{subfigure}{0.48\textwidth}
        \centering
        \includegraphics[width=\textwidth]{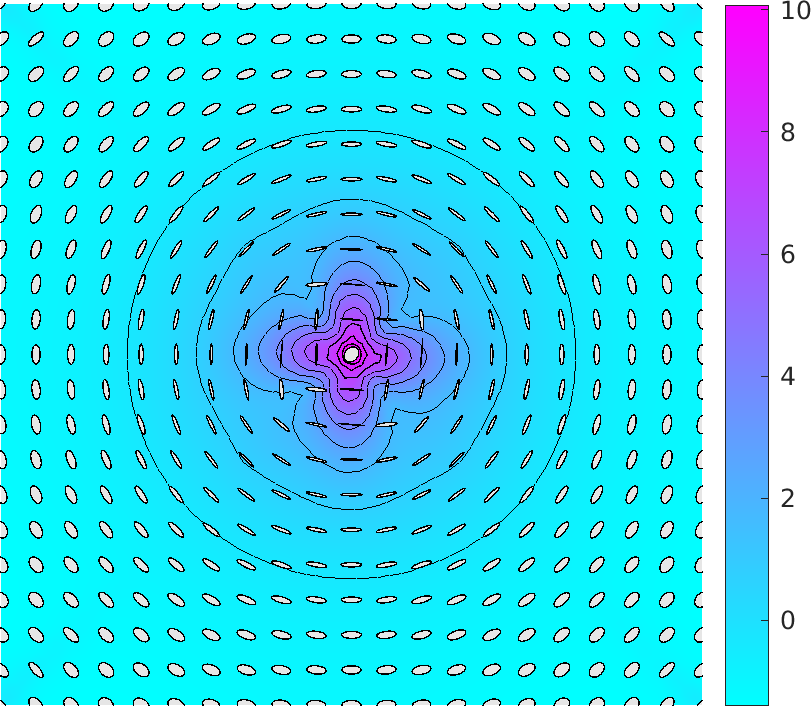}
        \caption{Ring $\mu_2$, $C(\mu_2,W^{(k)})^{-1}=0.9994$}
        \label{fig:optimalW_2}
    \end{subfigure}

    ~\\~\\

    \begin{subfigure}{0.48\textwidth}
        \centering
        \includegraphics[width=\textwidth]{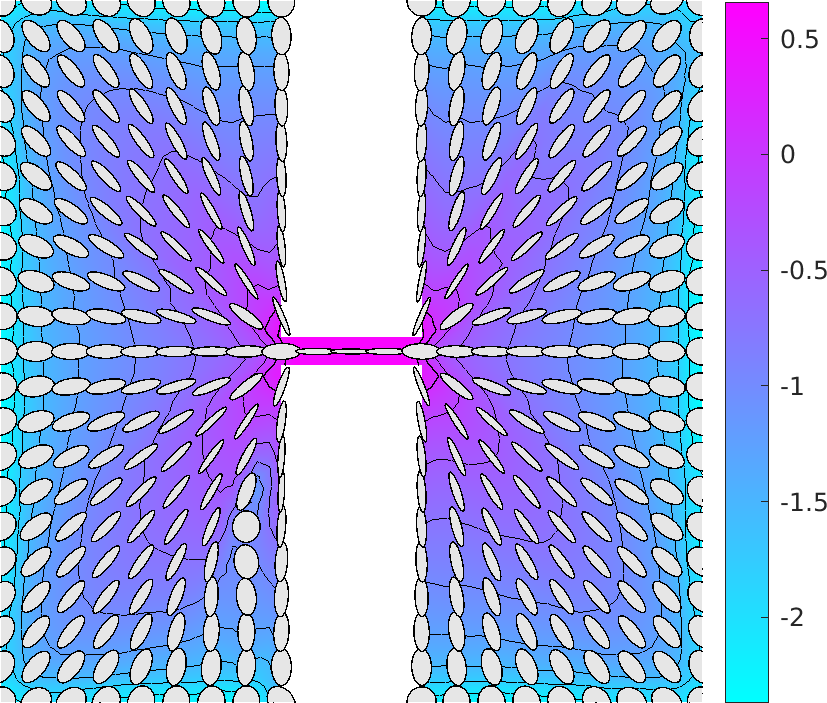}
        \caption{H-shaped $\mu_3$, $C(\mu_3,W^{(k)})^{-1}=0.9166$}
        \label{fig:optimalW_3}
    \end{subfigure}
    \begin{subfigure}{0.48\textwidth}
        \centering
        \includegraphics[width=\textwidth]{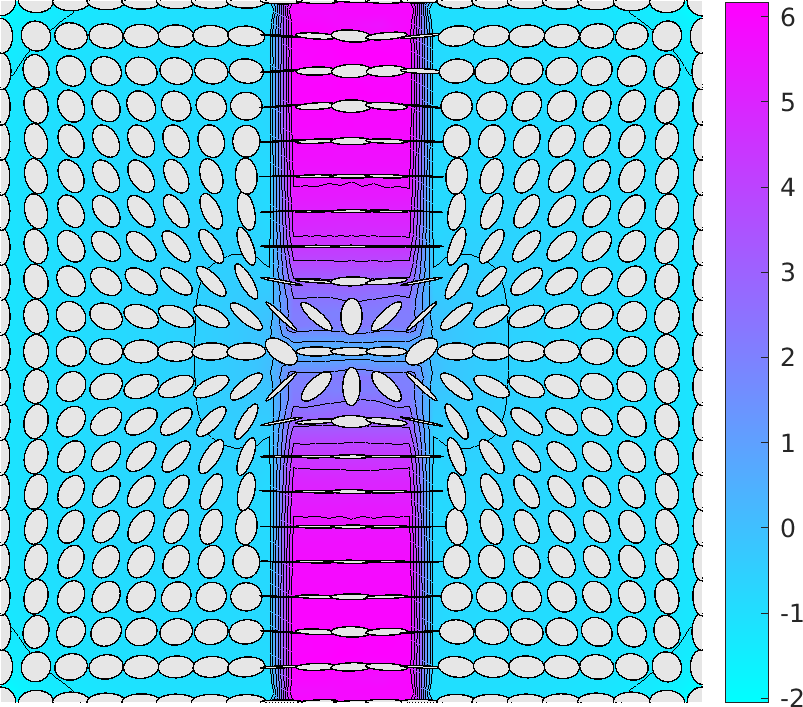}
        \caption{H-shaped $\mu_4$, $C(\mu_4,W^{(k)})^{-1}=0.9989$}
    \end{subfigure}
 \caption{Metric $W^{(k)}$ computed with $k=100$ iterations of Nesterov's acceleration.
 The colorbar corresponds to $\log_{10}(\trace(W^{(k)}))$, which is the amplitude of $W^{(k)}$ in log-scale. The ellipses represent the local anisotropy of $W^{(k)}$.}
    \label{fig:optimalW}
\end{figure}

\subsection{Application to the unadjusted Langevin algorithm (ULA)}

We consider now the Euler-Maruyama discretization of the Riemannian Langevin dynamic $\d X_t = (\div(W)- W\nabla V )\d t + \sqrt{2 W}\d B_t$ as in \eqref{eq:SDE_matrixWeighted}, which writes
\begin{equation}\label{eq:ULA}
 X_{n+1} = X_{n} +  (\div W(X_{n}) - W(X_{n})\nabla V(X_{n}) )\Delta t + \sqrt{2 \Delta t W(X_{n})} Z_{n},
\end{equation}
for some time step $\Delta t>0$. Here, $Z_1,Z_2,\dots $ are independent random vector drawn from $\mathcal{N}(0,I_d)$.
While Figure \ref{fig:optimalW} already shows the diffusion term $\sqrt{2 W}$ of the Riemannian Langevin SDE, Figure \ref{fig:drift} represents the drift term $\div(W)- W\nabla V$.
From it, we can observe how the new drift field connects different high-probability regions.
In particular, from Figure \ref{fig:drift_a}, we observe the non-preconditioned drift $\nabla V$ points towards the centres of the mixture components. While such drift pushes the particle $X_t$ towards high probability regions, it also prevents jumps from one mode to another. In comparison, the new drift field in Figure \ref{fig:drift_b} connects the different mixture components by pushing the particles to escape the modes, which intuitively improves the mixing speed of the SDE.
Likewise, in Figure \ref{fig:drift_d}, the drift $\nabla V$ only leads towards the high probability region but does not circulate along the ring. The drift field in Figure \ref{fig:drift_e} suggests leaving the high probability region to the centre of the ring, where it will be directed to other areas of the ring.
Regarding the uniform measure $\mu_3$ for which $\nabla V=0$, we see in Figure \ref{fig:drift_c} that the optimal metric $W$ creates a drift $\div W$ which pushes the particles towards the bottleneck of distribution so they can get to the other side. Regarding $\mu_4$, the cyan area in Figure \ref{fig:drift_f} provides additional leeway for the mixing to happen.

\begin{figure}
\centering
 \begin{subfigure}{0.32\textwidth}
        \centering
        \includegraphics[width=\textwidth]{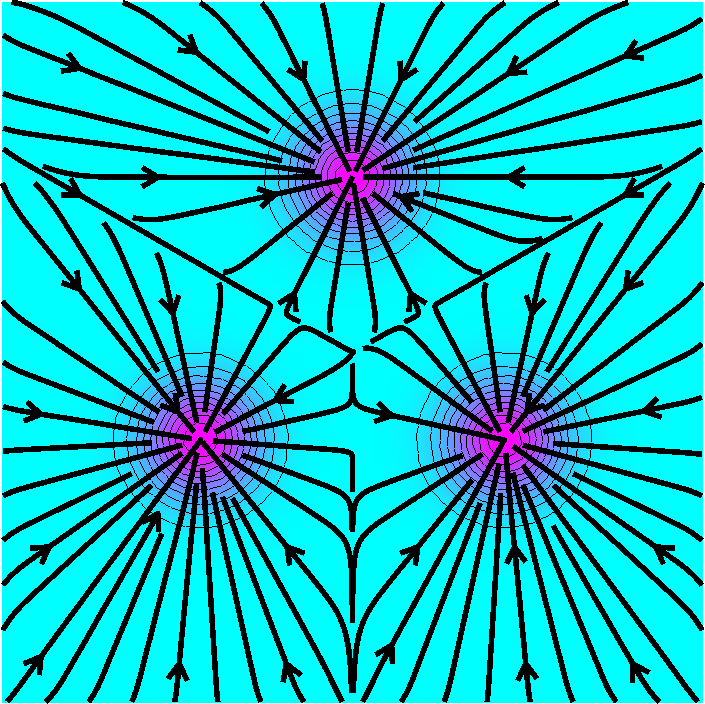}
        \caption{Tri-modal $\mu_1$, $W\propto I_d$}
        \label{fig:drift_a}
    \end{subfigure}
    \begin{subfigure}{0.32\textwidth}
        \centering
        \includegraphics[width=\textwidth]{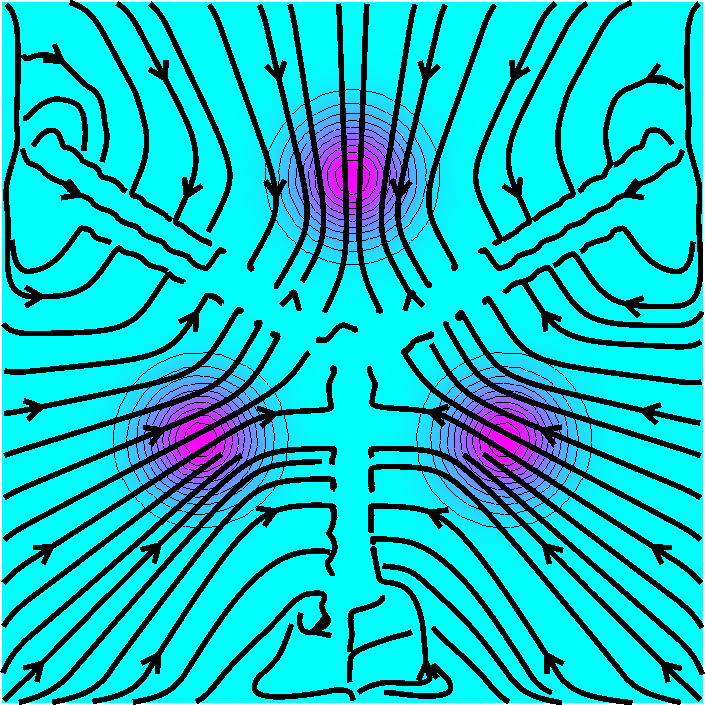}
        \caption{Tri-modal $\mu_1$, optimal $W$}
        \label{fig:drift_b}
    \end{subfigure}
    \begin{subfigure}{0.32\textwidth}
        \centering
        \includegraphics[width=\textwidth]{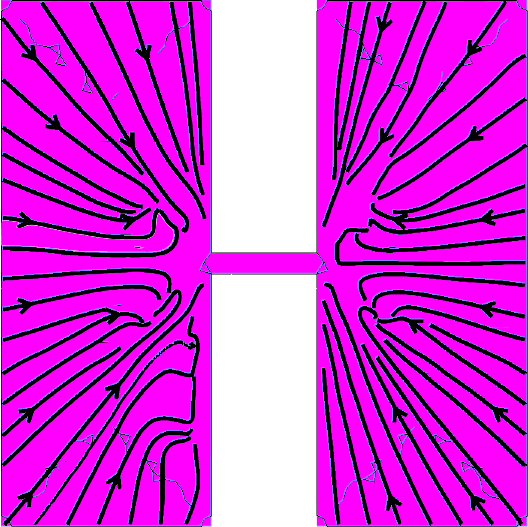}
        \caption{H-shaped $\mu_3$, optimal $W$}
        \label{fig:drift_c}
    \end{subfigure}

    ~\\~\\
    \begin{subfigure}{0.32\textwidth}
        \centering
        \includegraphics[width=\textwidth]{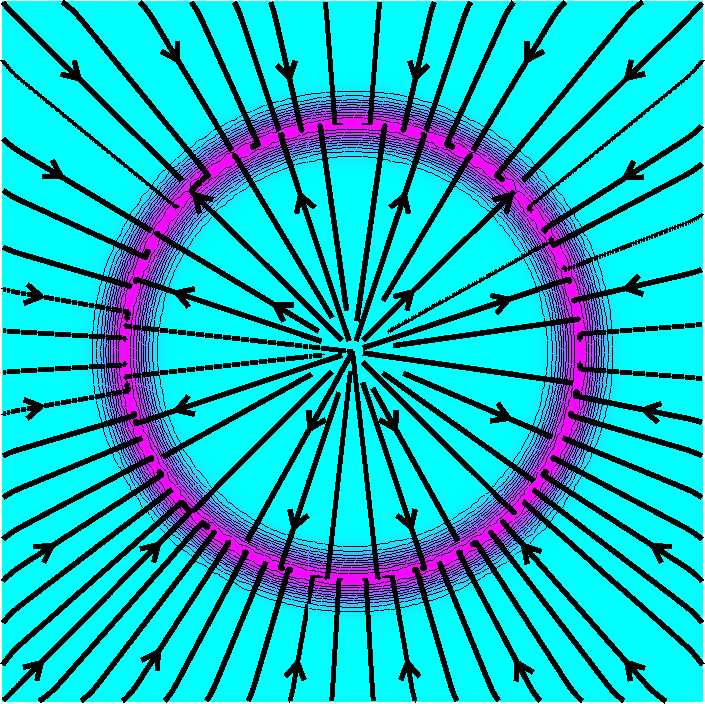}
        \caption{Ring $\mu_2$, $W\propto I_d$}
        \label{fig:drift_d}
    \end{subfigure}
    \begin{subfigure}{0.32\textwidth}
        \centering
        \includegraphics[width=\textwidth]{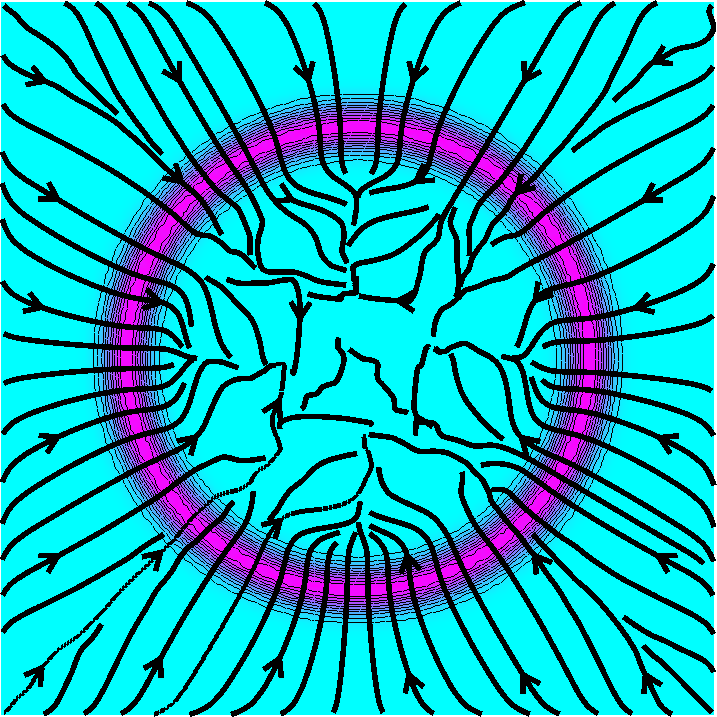}
        \caption{Ring $\mu_2$, optimal $W$}
        \label{fig:drift_e}
    \end{subfigure}
    \begin{subfigure}{0.32\textwidth}
        \centering
        \includegraphics[width=\textwidth]{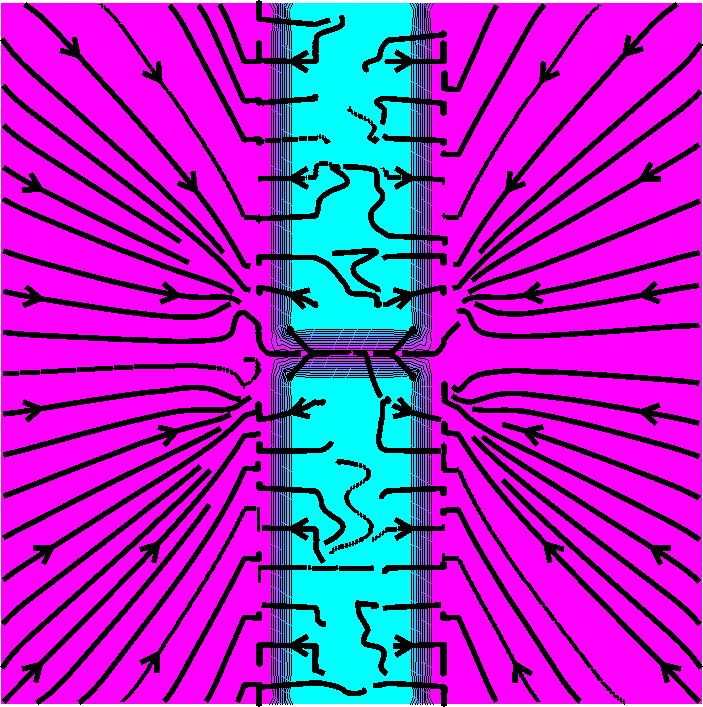}
        \caption{H-shaped $\mu_4$, optimal $W$}
        \label{fig:drift_f}
    \end{subfigure}

 \caption{Drift $(\div(W)- W\nabla V_i)$ for the four considered measures $\d\mu_i\propto\exp(-V_i)\d x$, $i=1,\hdots,4$, with either the constant metric $W\propto I_d$ (Figures \ref{fig:drift_a} and \ref{fig:drift_d}) or the optimal metric $W$.
  }
    \label{fig:drift}
\end{figure}

Finally, Figure \ref{fig:ULA} shows the particle trace of $5000$ iterations of the ULA algorithm \eqref{eq:ULA}. We observe that the unpreconditioned ULA (top row) has a large bias compared to the preconditioned ULA (middle row). In particular, the preconditioned ULA is more robust with respect to large time steps $\Delta t$ compared to the non-preconditioned ULA, see for instance the difference between Figure \ref{fig:ULA_b} and Figure \ref{fig:ULA_f}. For the ring measure $\mu_2$, having an anisotropic diffusion term (remember Figure \ref{fig:optimalW_2}) permits a better mixing of the particles.

\begin{figure}
\centering
 \begin{subfigure}{0.24\textwidth}
        \centering
        \includegraphics[width=\textwidth]{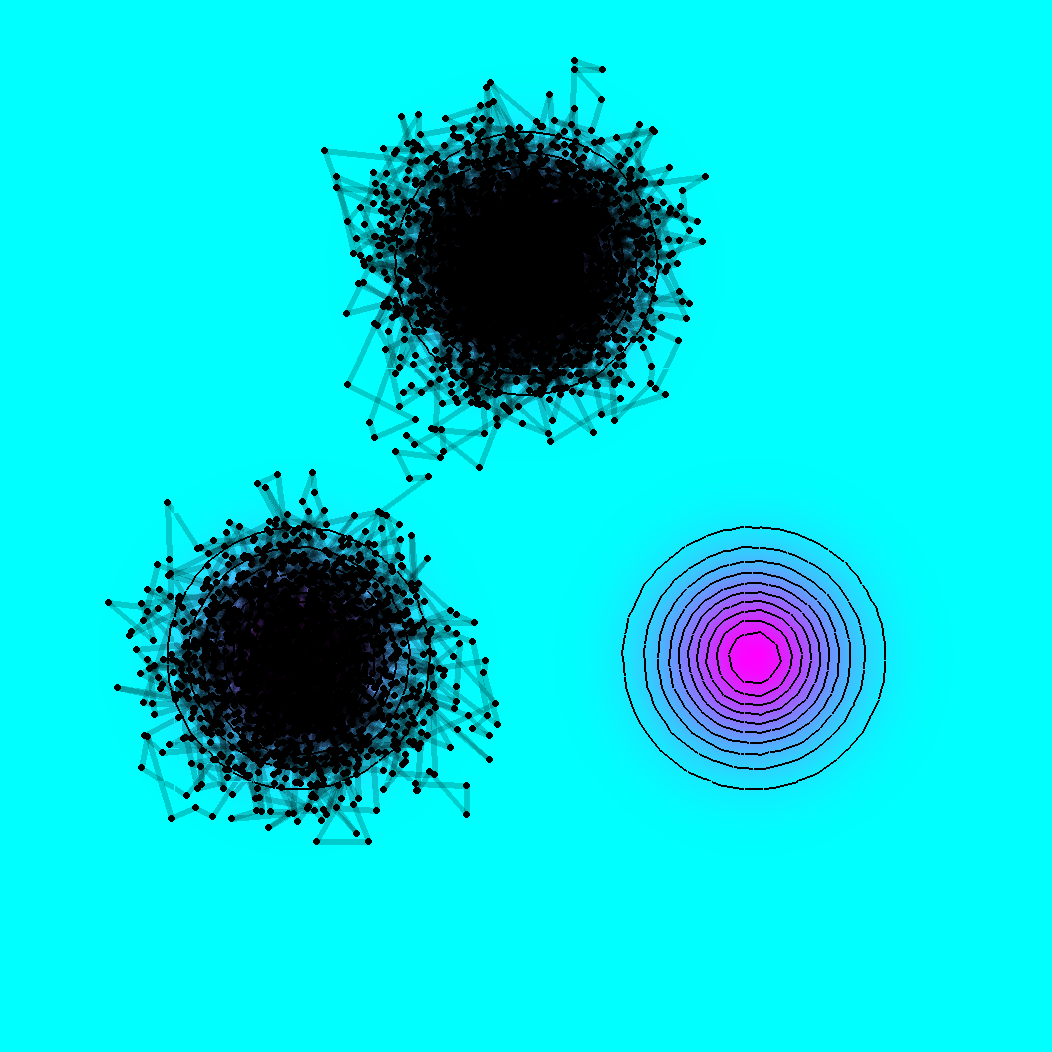}
        \caption{ $\Delta t = 0.015$}
        \label{fig:ULA_a}
    \end{subfigure}
    \begin{subfigure}{0.24\textwidth}
        \centering
        \includegraphics[width=\textwidth]{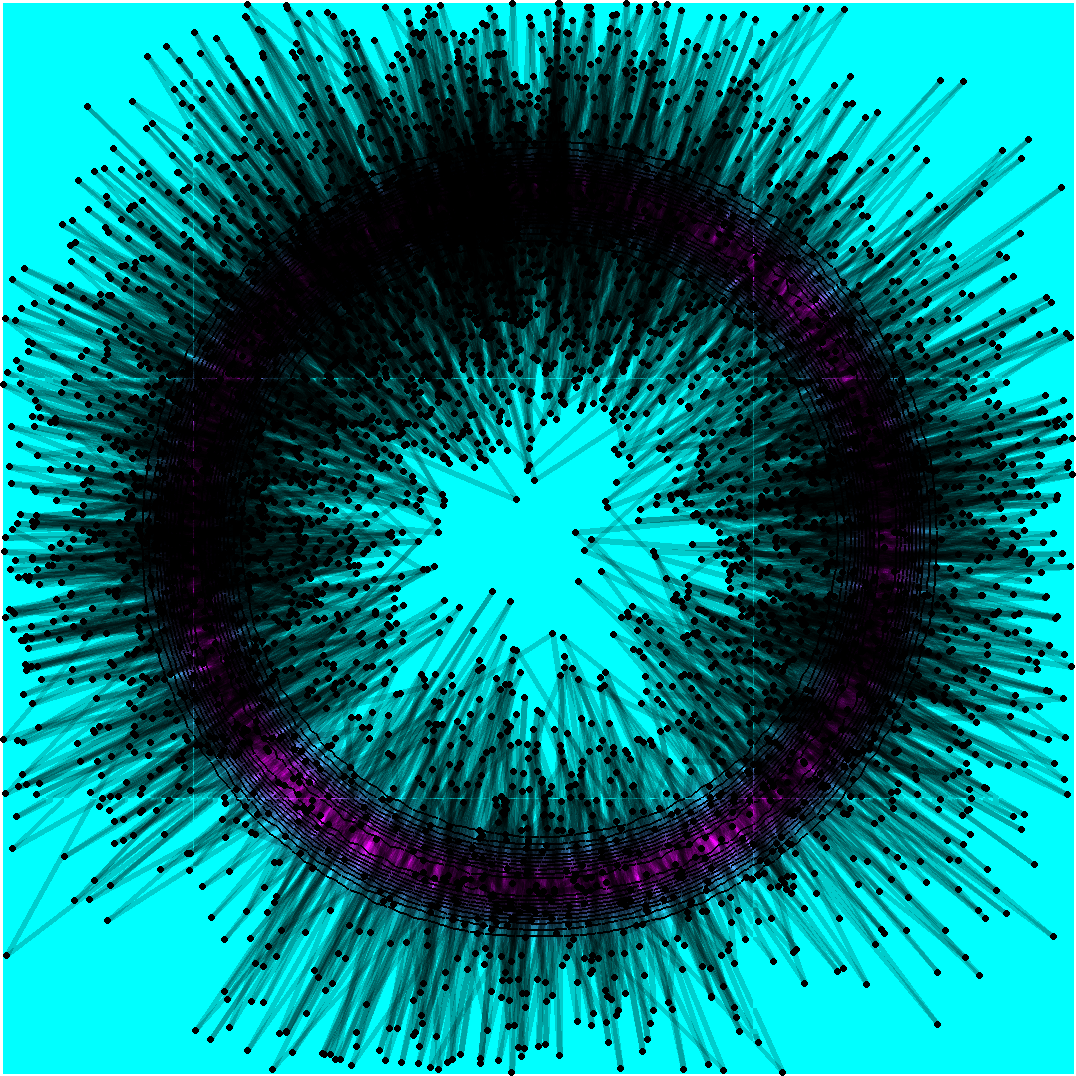}
        \caption{ $\Delta t = 0.015$}
        \label{fig:ULA_b}
    \end{subfigure}
    \begin{subfigure}{0.24\textwidth}
        \centering
        \includegraphics[width=\textwidth]{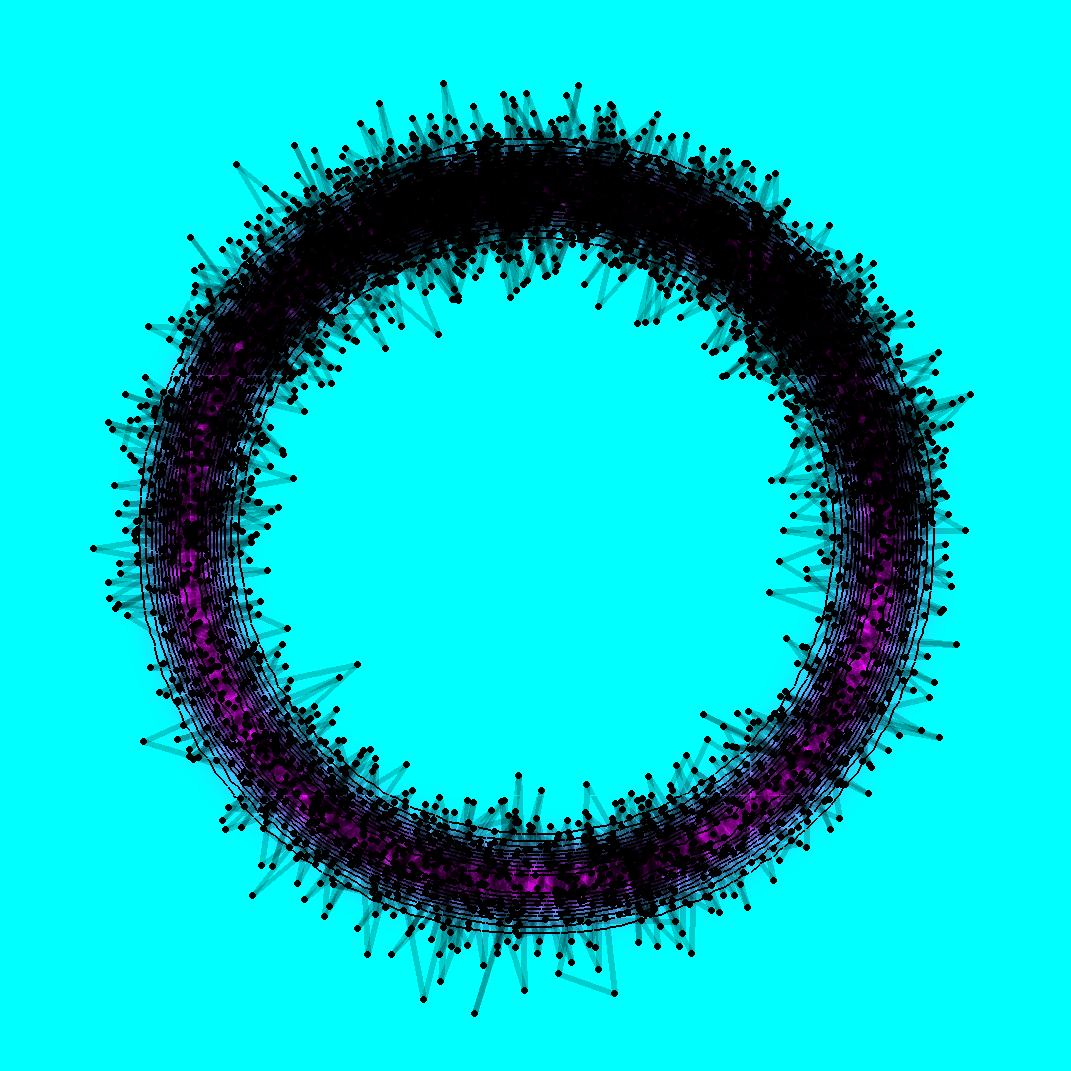}
        \caption{ $\Delta t = 0.01$}
        \label{fig:ULA_c}
    \end{subfigure}
    \begin{subfigure}{0.24\textwidth}
        \centering
        \includegraphics[width=\textwidth]{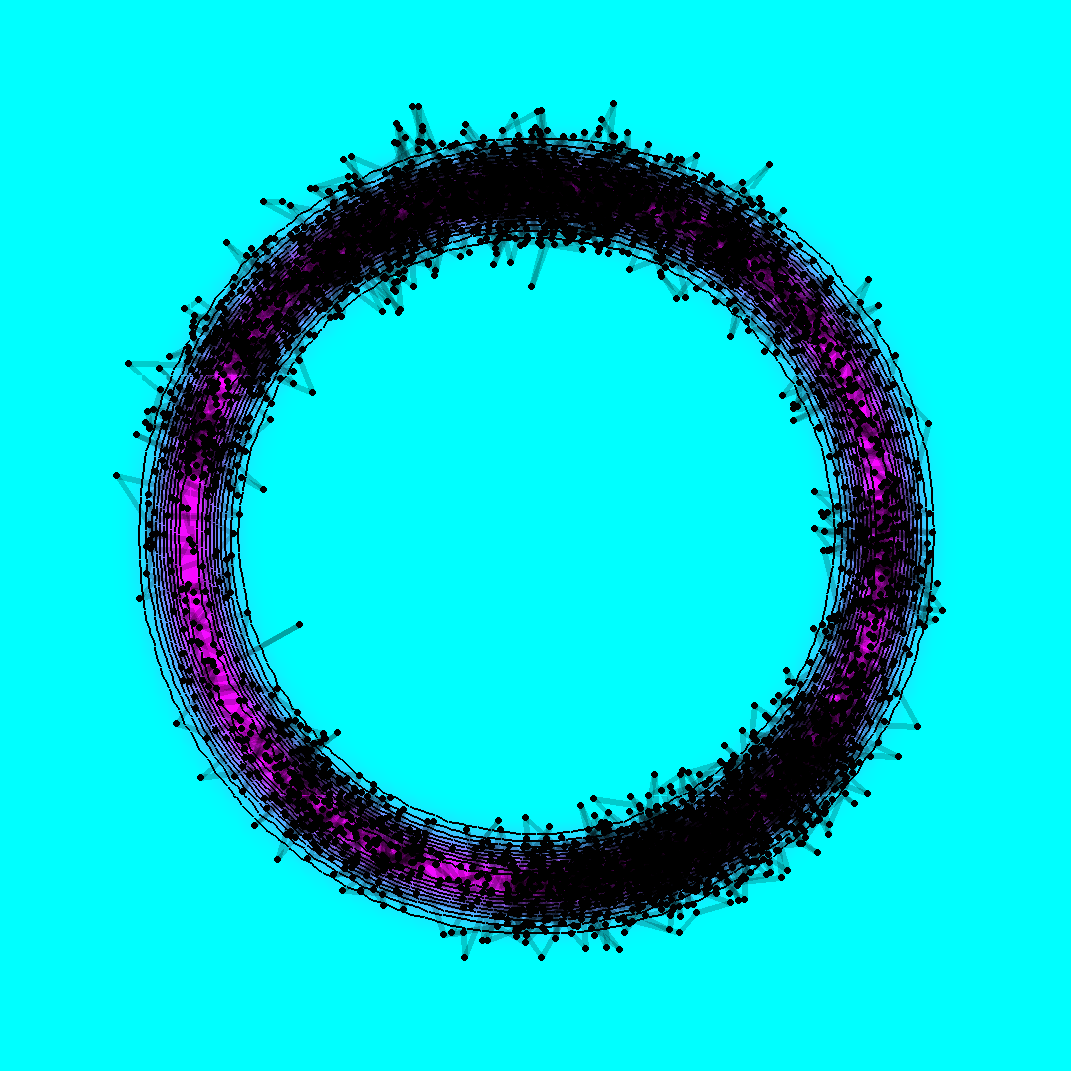}
        \caption{ $\Delta t = 0.005$}
        \label{fig:ULA_d}
    \end{subfigure}

    ~\\
    \begin{subfigure}{0.24\textwidth}
        \centering
        \includegraphics[width=\textwidth]{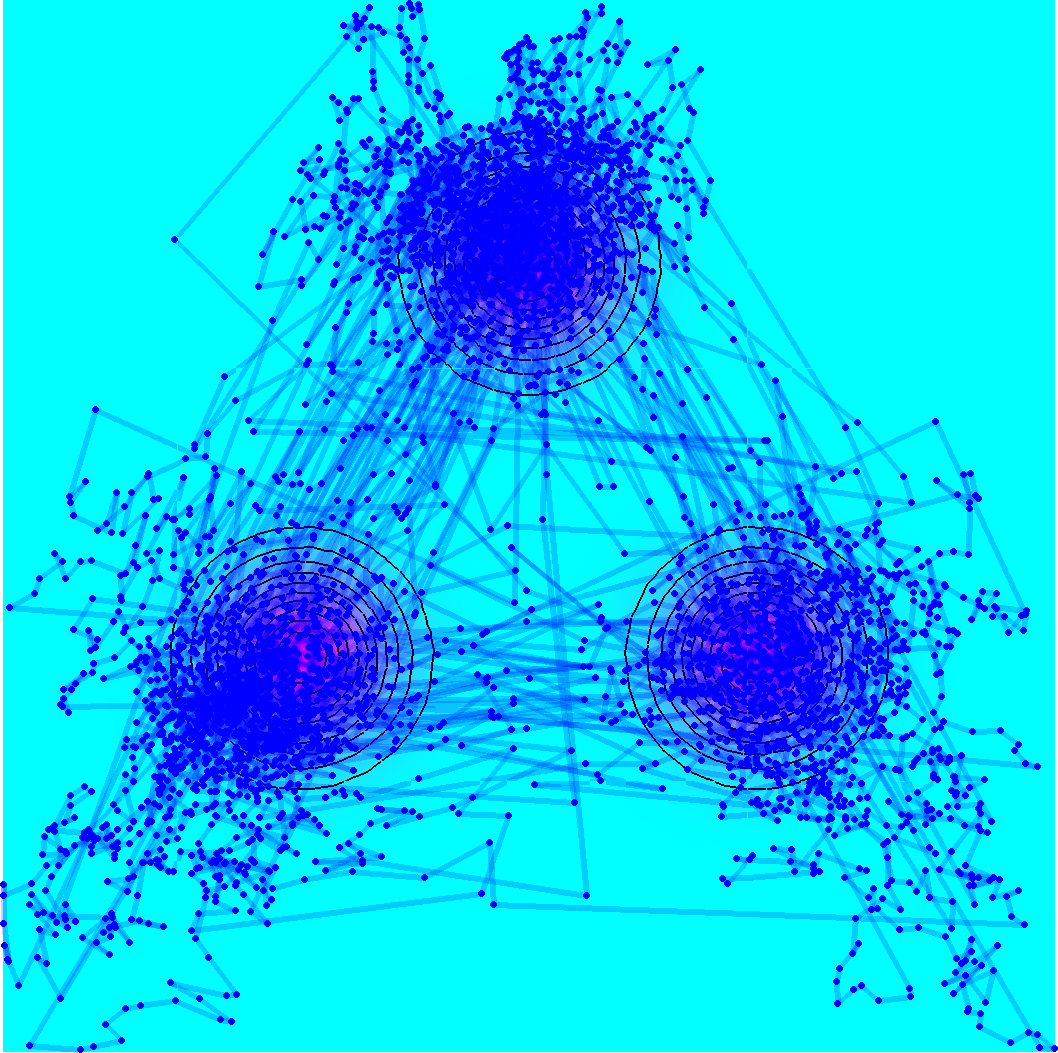}
        \caption{ $\Delta t = 0.015$}
        \label{fig:ULA_e}
    \end{subfigure}
    \begin{subfigure}{0.24\textwidth}
        \centering
        \includegraphics[width=\textwidth]{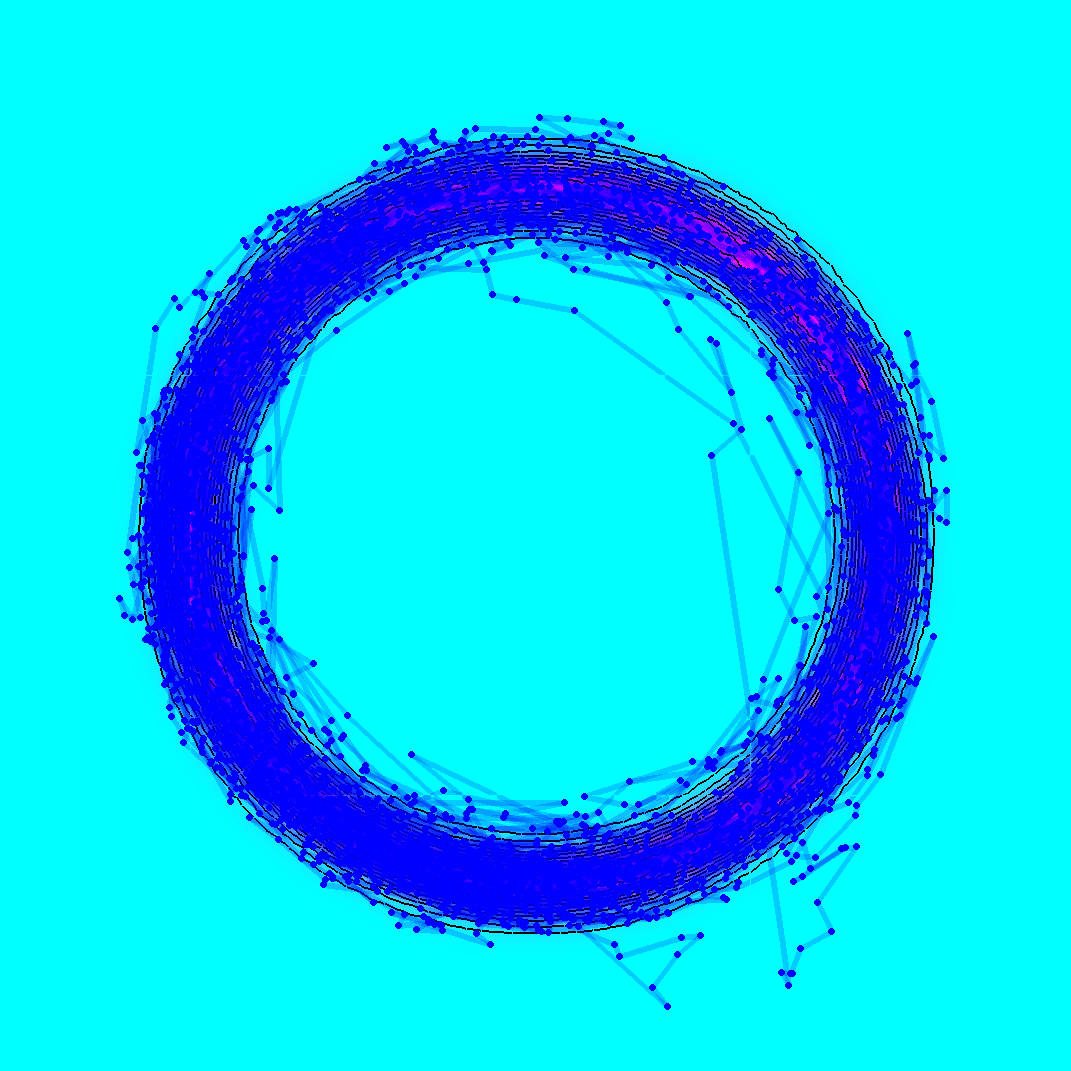}
        \caption{ $\Delta t = 0.015$}
        \label{fig:ULA_f}
    \end{subfigure}
    \begin{subfigure}{0.24\textwidth}
        \centering
        \includegraphics[width=\textwidth]{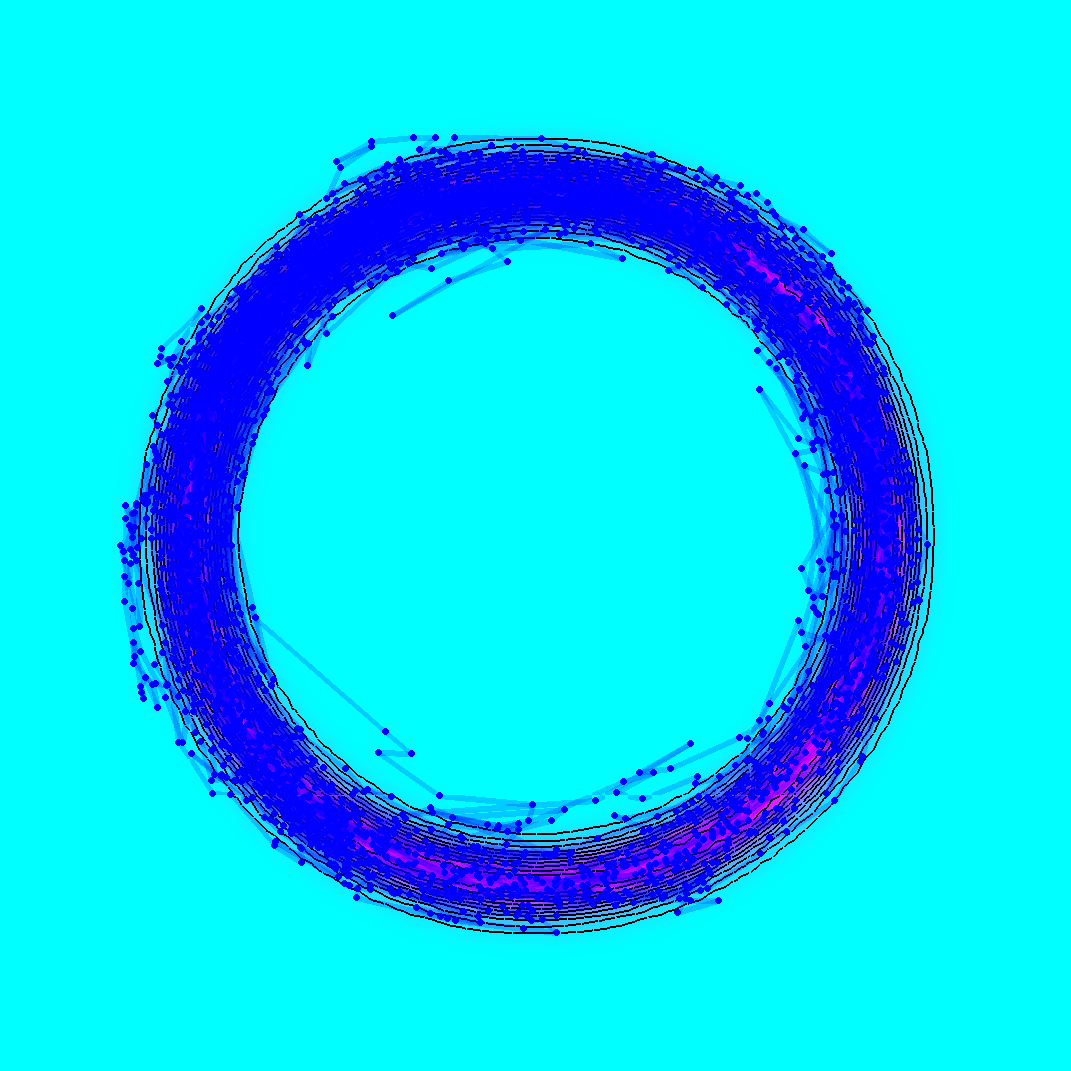}
        \caption{ $\Delta t = 0.01$}
        \label{fig:ULA_g}
    \end{subfigure}
    \begin{subfigure}{0.24\textwidth}
        \centering
        \includegraphics[width=\textwidth]{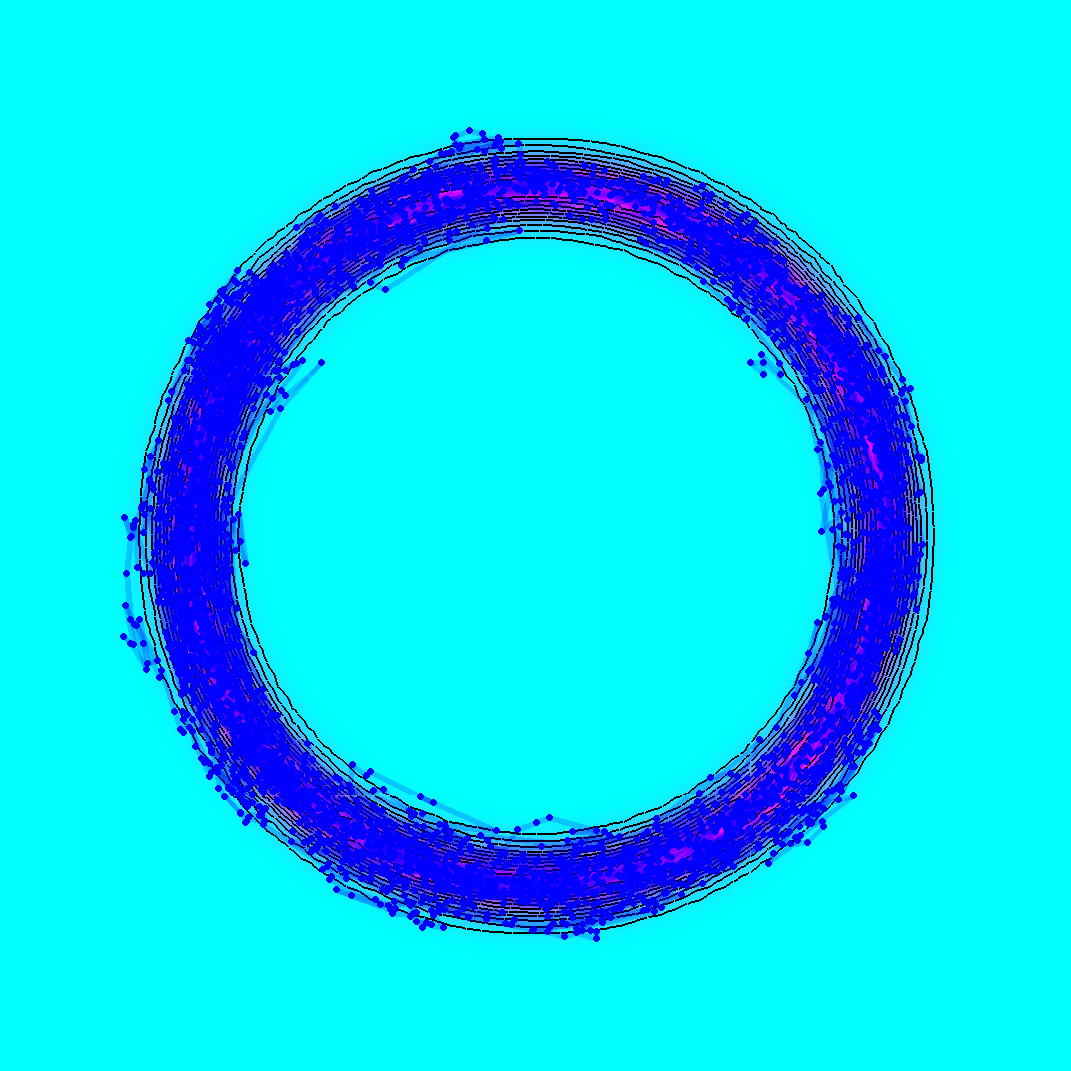}
        \caption{ $\Delta t = 0.005$}
        \label{fig:ULA_h}
    \end{subfigure}


 \caption{Trace of 5000 iterations of ULA \eqref{eq:ULA} for different $\Delta t$ using either the unpreconditioned ULA ($W\propto I_d$, top row) or the precondtionned ULA (optimal $W$, bottom row).}
    \label{fig:ULA}
\end{figure}

\section{Conclusion}
This paper considers adding a Riemannian metric to the Poincaré inequality and investigates the notion of optimal metric by minimizing the associated Poincaré constant. Several important properties of the optimal metric have been obtained and a gradient-based numerical algorithm is proposed to compute it. We numerically demonstrate our theory on four different 2D probability distributions, where the optimal matrix field reveals the geometric structure of these distributions. These metrics are then exploited to precondition the Langevin SDE.

While this work has shown various important properties and significant benefits of the optimal metric, much is left to be done, especially on the computational side. In the following, we elaborate on a few possible directions for future development.
\begin{itemize}

 \item Although Theorem \ref{th:WstartMomentMap} establishes the existence of an optimal metric, the issue of its uniqueness remains unresolved. Given that optimal metrics must be positive-semidefinite Stein kernels (see Theorem \ref{th:Stein}), one approach to investigate the uniqueness would be to ascertain whether positive-semidefinite Stein kernels are indeed unique.

 \item Stein kernels have proven effectiveness in constructing statistical tests to discern whether data samples originate from a particular distribution \cite{liu2016kernelized}. Assessing whether the positive-semidefinite Stein kernel we obtain via the optimal metric improves upon existing statistical tests is of natural interest.

 \item While Algorithm \ref{algo} permits accurate computation of the optimal metric, its application to large-scale problems requires further investigations. In particular, the finite element methods we employed become prominently expensive if the underlying dimension $d$ is high, which makes it inapplicable for high-dimension problems. Kernel-based methods as in \cite{pillaud2020statistical,pillaud2023kernelized} or neural networks in the spirit of \cite{belkacemi2023autoencoders} may be able to offer a viable to tackle the dimensionality issue.

 \item An additional concern regarding Algorithm \ref{algo} is its dependency on global knowledge of $\mu$, which may not be readily available beforehand. Algorithms capable of estimating $W$ using available data (e.g., a few samples drawn from $\mu$ or some moments estimates) or dynamically adapting to new information (e.g., during an MCMC procedure that samples from $\mu$) would present more compelling approaches.

\end{itemize}

\begin{appendix}
\section{Proof of Proposition \ref{prop:EquivalentSDE_WPI}}\label{proof:EquivalentSDE_WPI}

 The Fokker-Planck equation satisfied by the density $\mu_t\propto \exp(-V_t)$ of $X_t$ solution to \eqref{eq:SDE_matrixWeighted} writes
 \begin{align}
  \partial_t \mu_t
  &= - \nabla\cdot ( (\div(W)- W\nabla V ) \mu_t ) + \nabla\cdot( \div(W \mu_t) ) \nonumber\\
  &= - \nabla\cdot ( (\div(W)- W\nabla V ) \mu_t ) + \nabla\cdot( \div(W)\mu_t+ W \nabla \mu_t ) \nonumber\\
  &= - \nabla\cdot ( - (W\nabla V)  \mu_t -   W \nabla \mu_t ) \nonumber\\
  &= -\nabla\cdot ( \mu_t W (\nabla V_t-\nabla V)  ). \label{eq:FokkerPlanck}
 \end{align}
 Because $ \chi^2(\mu_t,\mu) = \int (\frac{\mu_t}{\mu})^2 \d\mu -1$ we can write
 \begin{align}
  \partial_t \chi^2(\mu_t,\mu)
  &=  2 \int  \partial_t \left(\frac{\mu_t}{\mu}\right) \mu_t \d x \nonumber\\
  &\overset{\eqref{eq:FokkerPlanck}}{=} -  2 \int  \nabla\cdot ( \mu_t W (\nabla V_t-\nabla V)  ) \left(\frac{\mu_t}{\mu}\right) \d x \nonumber\\
  &= 2 \int  ( \mu_t W (\nabla V_t-\nabla V)  )^\top  \nabla \left(\frac{\mu_t}{\mu}\right) \d x \nonumber\\
  &= -2 \int   \mu_t \nabla \log\left(\frac{\mu_t}{\mu}\right)^\top  W \nabla \left(\frac{\mu_t}{\mu}\right) \d x \nonumber\\
  &= -2 \int    \nabla \left(\frac{\mu_t}{\mu}\right)^\top  W \nabla \left(\frac{\mu_t}{\mu}\right) \d\mu \label{eq:tmp789632}
 \end{align}
 To show that $C(\mu,W)\leq C$ implies \eqref{eq:EquivalentSDE_WPI}, we apply the Riemannian Poincaré inequality \eqref{eq:Poincare_Riemannian} to \eqref{eq:tmp789632} to obtain
 \begin{align}
  \partial_t \chi^2(\mu_t,\mu)
  &\overset{\eqref{eq:Poincare_Riemannian}}{\leq} -\frac{2}{C} \Var_\mu( \mu_t/\mu )
  = -\frac{2}{C} \chi^2(\mu_t,\mu).
 \end{align}
 This shows that the function $t\mapsto e^{2t/C}\chi^2(\mu_t,\mu)$ is decreasing (because its derivative $2/C e^{2t/C} \chi^2(\mu_t,\mu)+ e^{2t/C}\partial_t \chi^2(\mu_t,\mu)\leq0$ is nonpositive) so that $\chi^2(\mu_0,\mu)\geq e^{2t/C}\chi^2(\mu_t,\mu)$, which is \eqref{eq:EquivalentSDE_WPI}.

 We now show that if \eqref{eq:EquivalentSDE_WPI} holds for any probability measure $\mu_0$ then $C(\mu,W)\leq C$.
 Let $f_0$ be a positive function such that $\int f_0 \d\mu=1$. Letting $\d\mu_0 = f_0 \d\mu$, inequality \eqref{eq:EquivalentSDE_WPI} becomes  $\chi^2(\mu_t,\mu)\leq e^{-2t/C} \Var_\mu(f_0) $.  A Taylor expansion around $t=0$ yields
 \begin{equation}\label{eq:tmp78960}
  \chi^2(\mu_0,\mu)+ \partial_t \chi^2(\mu_t,\mu)_{|t=0}  t + \mathcal{O}(t^2) \leq (1-2t/C) \Var_\mu(f_0) + \mathcal{O}(t^2),
 \end{equation}
 for any $t\geq0$.
 By \eqref{eq:tmp789632} we have $\partial_t \chi^2(\mu_t,\mu)_{|t=0}=-2 \int  \nabla f_0^\top  W \nabla f_0 \d\mu $ so that \eqref{eq:tmp78960} yields
 \begin{equation}\label{eq:tmp257682}
   \Var_\mu(f_0) \leq  C \int  \nabla f_0^\top  W \nabla f_0 \d\mu,
 \end{equation}
 for any positive function $f_0$ such that $\int f_0\d\mu=1$.
 It remains to show that it also holds for any smooth function $f$ with $\int  \nabla f^\top  W \nabla f \d\mu <\infty$.
 If such a $f$ is lower bounded, then $f_0=(f-\inf f)/\alpha$ with $\alpha=\int f\d\mu  -\inf f$ is a positive function with $\int f_0\d\mu=1$ and so \eqref{eq:tmp257682} yields $\alpha^{-2}\Var_\mu(f) \leq  \alpha^{-2} C \int  \nabla f^\top  W \nabla f \d\mu$. We conclude that any smooth lower-bounded $f$ satisfies the Riemannian Poincaré inequality.
 Finally, if $f$ is not lower bounded, we introduce the sequence of lower-bounded function $f_n = \max\{f;-n\}$ so that
 $$
  \Var_\mu(f_n) \leq   C \int  \nabla f_n^\top  W \nabla f_n \d\mu .
 $$
 Taking the limit when $n\rightarrow\infty$ we obtain $\Var_\mu(f) \leq   C \int  \nabla f^\top  W \nabla f \d\mu$.

\section{Proof of Corollary \ref{cor:CLT}}\label{proof:CLT}

The proof is a simple adaptation of the proof of Theorem 3.3 in \cite{fathi2019stein}.
By Theorem \ref{th:Stein}, $W_N$ defined in \eqref{eq:Wiid} is a Stein kernel so that
\begin{align*}
 S_p(\mu_N|\gamma)^p
 &\leq  \int \left\|  \E\left[\left. \sum_{i=1}^N \frac{W(X_i)-I_d}{N} \right|\overline{X}_N \right] \right\|_F^p \d\mu_N \\
 \text{(Jensen's inequality) }&\leq  \int \E\left[\left. \left\|   \sum_{i=1}^N \frac{W(X_i)-I_d}{N}  \right\|_F^p \right|\overline{X}_N \right] \d\mu_N \\
 &=   \E\left[ \left\|   \sum_{i=1}^N \frac{W(X_i)-I_d}{N}  \right\|_F^p \right],
\end{align*}
Letting $Y_{ikl} = W_{kl}(X_i)-\delta_{kl}$, we deduce
\begin{align*}
S_p(\mu_N|\gamma)^p
&\leq \E\left[ \left(  \frac{1}{d^2}  \sum_{kl=1}^d d^2\left( \frac{1}{N}\sum_{i=1}^N  Y_{ikl} \right)^2  \right)^{p/2} \right] \\
\text{(Jensen's inequality) }&\leq \frac{1}{d^2}  \sum_{kl=1}^d \E\left[ \left(  d^2\left( \frac{1}{N}\sum_{i=1}^N  Y_{ikl} \right)^2  \right)^{p/2} \right] \\
&=d^{p-2} N^{-p} \sum_{kl=1}^d \E\left[ \left(  \sum_{i=1}^N  Y_{ikl} \right)^p \right] \\
\text{(Rosenthal) }&\leq d^{p-2} N^{-p} \sum_{kl=1}^d K_p \max\left\{  \sum_{i=1}^N \E[Y_{ikl}^{p}]; \left(\sum_{i=1}^N\E[Y_{ikl}^2 ]\right)^{p/2} \right\}
\end{align*}
where we applied the Rosenthal inequality \cite{ibragimov2002exact} in the last step. Here, $K_p$ is a constant depending only on $p$.
On one side we have $\sum_{i=1}^N \E[Y_{ikl}^{p}] = N \E\left[  (W_{kl}(X)-\delta_{kl})^{p}\right]$ and, on the other side, a Hölder inequality yields
\begin{align*}
 \left(\sum_{i=1}^N\E[Y_{ikl}^2 ]\right)^{p/2}
 =N^{p/2} \E[ (W_{kl}(X)-\delta_{kl})^2 ] ^{p/2}
 \leq N^{p/2} \E[ (W_{kl}(X)-\delta_{kl})^p ] .
\end{align*}
Thus we deduce
\begin{align*}
S_p(\mu_N|\gamma)^p
&\leq d^{p-2} N^{-p} \sum_{kl=1}^d K_p \E[  (W_{kl}(X)-\delta_{kl})^{p}] \max\left\{  N  ; N^{p/2}   \right\} \\
&= d^{p-2} N^{-p/2}   K_p \int \sum_{kl=1}^d   (W_{kl}(X)-\delta_{kl})^{p} \d\mu
\end{align*}
Together with Proposition 3.1 in \cite{fathi2019stein}, which  establishes $\mathcal{W}_p(\mu_N,\gamma)\leq C_p' S_p(\mu_N|\gamma)$ for another constant $C_p'$ depending only on $p$, we conclude the proof.

\end{appendix}

{
\bibliographystyle{siam}
\bibliography{references}

\begin{thebibliography}{10}

\bibitem{acosta2004optimal}
{\sc G.~Acosta and R.~Dur{\'a}n}, {\em An optimal poincar{\'e} inequality in
  l$^1$ for convex domains}, Proceedings of the american mathematical society,
  132 (2004), pp.~195--202.

\bibitem{allaire2007numerical}
{\sc G.~Allaire}, {\em Numerical analysis and optimization: an introduction to
  mathematical modelling and numerical simulation}, OUP Oxford, 2007.

\bibitem{andrieu2022comparison}
{\sc C.~Andrieu, A.~Lee, S.~Power, and A.~Q. Wang}, {\em Comparison of markov
  chains via weak poincar{\'e} inequalities with application to pseudo-marginal
  mcmc}, The Annals of Statistics, 50 (2022), pp.~3592--3618.

\bibitem{andrieu2022poincar}
\leavevmode\vrule height 2pt depth -1.6pt width 23pt, {\em Poincar\'e
  inequalities for markov chains: a meeting with cheeger, lyapunov and
  metropolis}, arXiv preprint arXiv:2208.05239,  (2022).

\bibitem{bakry2008simple}
{\sc D.~Bakry, F.~Barthe, P.~Cattiaux, and A.~Guillin}, {\em {A simple proof of
  the Poincaré inequality for a large class of probability measures}},
  Electronic Communications in Probability, 13 (2008), pp.~60 -- 66.

\bibitem{bakry2006diffusions}
{\sc D.~Bakry and M.~{\'E}mery}, {\em Diffusions hypercontractives}, in
  S{\'e}minaire de Probabilit{\'e}s XIX 1983/84: Proceedings, Springer, 2006,
  pp.~177--206.

\bibitem{bakry2014analysis}
{\sc D.~Bakry, I.~Gentil, M.~Ledoux, et~al.}, {\em Analysis and geometry of
  Markov diffusion operators}, vol.~103, Springer, 2014.

\bibitem{barbour2005introduction}
{\sc A.~D. Barbour and L.~H.~Y. Chen}, {\em An introduction to Stein's method},
  vol.~4, World Scientific, 2005.

\bibitem{bebendorf2003note}
{\sc M.~Bebendorf}, {\em A note on the poincar{\'e} inequality for convex
  domains}, Zeitschrift f{\"u}r Analysis und ihre Anwendungen, 22 (2003),
  pp.~751--756.

\bibitem{belkacemi2023autoencoders}
{\sc Z.~Belkacemi, M.~Bianciotto, H.~Minoux, T.~Leli{\`e}vre, G.~Stoltz, and
  P.~Gkeka}, {\em Autoencoders for dimensionality reduction in molecular
  dynamics: Collective variable dimension, biasing, and transition states}, The
  Journal of Chemical Physics, 159 (2023).

\bibitem{berman2013real}
{\sc R.~J. Berman and B.~Berndtsson}, {\em Real monge-amp{\`e}re equations and
  k{\"a}hler-ricci solitons on toric log fano varieties}, in Annales de la
  Facult{\'e} des sciences de Toulouse: Math{\'e}matiques, vol.~22, 2013,
  pp.~649--711.

\bibitem{bobkov1996variance}
{\sc S.~G. Bobkov and C.~Houdr{\'e}}, {\em Variance of lipschitz functions and
  an isoperimetric problem for a class of product measures}, Bernoulli,
  (1996), pp.~249--255.

\bibitem{bobkov2000brunn}
{\sc S.~G. Bobkov and M.~Ledoux}, {\em From brunn-minkowski to brascamp-lieb
  and to logarithmic sobolev inequalities}, Geometric and Functional Analysis,
  10 (2000), pp.~1028--1052.

\bibitem{bobkov2009weighted}
\leavevmode\vrule height 2pt depth -1.6pt width 23pt, {\em Weighted
  poincar{\'e}-type inequalities for cauchy and other convex measures}, The
  Annals of Probability, 37 (2009), pp.~403--427.

\bibitem{bonnefont2016note}
{\sc M.~Bonnefont, A.~Joulin, and Y.~Ma}, {\em A note on spectral gap and
  weighted poincar{\'e} inequalities for some one-dimensional diffusions},
  ESAIM: Probability and Statistics, 20 (2016), pp.~18--29.

\bibitem{bonnefont2016spectral}
\leavevmode\vrule height 2pt depth -1.6pt width 23pt, {\em Spectral gap for
  spherically symmetric log-concave probability measures, and beyond}, Journal
  of Functional Analysis, 270 (2016), pp.~2456--2482.

\bibitem{boyd2004convex}
{\sc S.~Boyd, S.~P. Boyd, and L.~Vandenberghe}, {\em Convex optimization},
  Cambridge university press, 2004.

\bibitem{brascamp1976extensions}
{\sc H.~J. Brascamp and E.~H. Lieb}, {\em On extensions of the brunn-minkowski
  and pr{\'e}kopa-leindler theorems, including inequalities for log concave
  functions, and with an application to the diffusion equation}, Journal of
  functional analysis, 22 (1976), pp.~366--389.

\bibitem{brezis2011functional}
{\sc H.~Brezis and H.~Br{\'e}zis}, {\em Functional analysis, Sobolev spaces and
  partial differential equations}, vol.~2, Springer, 2011.

\bibitem{cacoullos1982upper}
{\sc T.~Cacoullos}, {\em On upper and lower bounds for the variance of a
  function of a random variable}, The Annals of Probability, 10 (1982),
  pp.~799--809.

\bibitem{cacoullos1994variational}
{\sc T.~Cacoullos, V.~Papathanasiou, and S.~A. Utev}, {\em Variational
  inequalities with examples and an application to the central limit theorem},
  The Annals of Probability,  (1994), pp.~1607--1618.

\bibitem{cheeger1999differentiability}
{\sc J.~Cheeger}, {\em Differentiability of lipschitz functions on metric
  measure spaces}, Geometric \& Functional Analysis GAFA, 9 (1999),
  pp.~428--517.

\bibitem{chen2021stein}
{\sc L.~H. Chen}, {\em Stein’s method of normal approximation: Some
  recollections and reflections}, The Annals of Statistics, 49 (2021),
  pp.~1850--1863.

\bibitem{cheng2024fast}
{\sc X.~Cheng, B.~Wang, J.~Zhang, and Y.~Zhu}, {\em Fast conditional mixing of
  mcmc algorithms for non-log-concave distributions}, Advances in Neural
  Information Processing Systems, 36 (2024).

\bibitem{chernozhuokov2022improved}
{\sc V.~Chernozhuokov, D.~Chetverikov, K.~Kato, and Y.~Koike}, {\em Improved
  central limit theorem and bootstrap approximations in high dimensions}, The
  Annals of Statistics, 50 (2022), pp.~2562--2586.

\bibitem{chewi2024analysis}
{\sc S.~Chewi, M.~A. Erdogdu, M.~Li, R.~Shen, and M.~S. Zhang}, {\em Analysis
  of langevin monte carlo from poincare to log-sobolev}, Foundations of
  Computational Mathematics,  (2024), pp.~1--51.

\bibitem{chewi2020exponential}
{\sc S.~Chewi, T.~Le~Gouic, C.~Lu, T.~Maunu, P.~Rigollet, and A.~Stromme}, {\em
  Exponential ergodicity of mirror-langevin diffusions}, Advances in Neural
  Information Processing Systems, 33 (2020), pp.~19573--19585.

\bibitem{clarke2013functional}
{\sc F.~Clarke}, {\em Functional analysis, calculus of variations and optimal
  control}, vol.~264, Springer, 2013.

\bibitem{cordero2015moment}
{\sc D.~Cordero-Erausquin and B.~Klartag}, {\em Moment measures}, Journal of
  Functional Analysis, 268 (2015), pp.~3834--3866.

\bibitem{courtade2019existence}
{\sc T.~A. Courtade, M.~Fathi, and A.~Pananjady}, {\em Existence of stein
  kernels under a spectral gap, and discrepancy bounds}, in Annales de
  l'Institut Henri Poincar{\'e}, Probabilit{\'e}s et Statistiques, vol.~55,
  Institut Henri Poincar{\'e}, 2019, pp.~777--790.

\bibitem{cui2022prior}
{\sc T.~Cui, X.~T. Tong, and O.~Zahm}, {\em Prior normalization for certified
  likelihood-informed subspace detection of bayesian inverse problems}, Inverse
  Problems, 38 (2022), p.~124002.

\bibitem{engquist2024adaptive}
{\sc B.~Engquist, K.~Ren, and Y.~Yang}, {\em Adaptive state-dependent diffusion
  for derivative-free optimization}, Communications on Applied Mathematics and
  Computation, 6 (2024), pp.~1241--1269.

\bibitem{engquist2024sampling}
\leavevmode\vrule height 2pt depth -1.6pt width 23pt, {\em Sampling with
  adaptive variance for multimodal distributions}, arXiv preprint
  arXiv:2411.15220,  (2024).

\bibitem{fang2021high}
{\sc X.~Fang and Y.~Koike}, {\em High-dimensional central limit theorems by
  stein’s method}, The Annals of Applied Probability, 31 (2021),
  pp.~1660--1686.

\bibitem{fathi2019stein}
{\sc M.~Fathi}, {\em Stein kernels and moment maps}, The Annals of Probability,
  47 (2019), pp.~2172--2185.

\bibitem{figalli2017monge}
{\sc A.~Figalli}, {\em The Monge--Amp{\`e}re equation and its applications},
  2017.

\bibitem{flock2023certified}
{\sc R.~Flock, Y.~Dong, F.~Uribe, and O.~Zahm}, {\em Certified coordinate
  selection for high-dimensional bayesian inversion with laplace prior},
  Statistics and Computing, 34 (2024), p.~134.

\bibitem{girolami2011riemann}
{\sc M.~Girolami and B.~Calderhead}, {\em Riemann manifold langevin and
  hamiltonian monte carlo methods}, Journal of the Royal Statistical Society
  Series B: Statistical Methodology, 73 (2011), pp.~123--214.

\bibitem{heredia2024one}
{\sc D.~Heredia, A.~Joulin, and O.~Roustant}, {\em On one dimensional weighted
  poincare inequalities for global sensitivity analysis}, arXiv preprint
  arXiv:2412.04918,  (2024).

\bibitem{ibragimov2002exact}
{\sc R.~Ibragimov and S.~Sharakhmetov}, {\em The exact constant in the
  rosenthal inequality for random variables with mean zero}, Theory of
  Probability \& Its Applications, 46 (2002), pp.~127--132.

\bibitem{klartag2013poincare}
{\sc B.~Klartag}, {\em Poincar{\'e} inequalities and moment maps}, in Annales
  de la Facult{\'e} des sciences de Toulouse: Math{\'e}matiques, vol.~22, 2013,
  pp.~1--41.

\bibitem{klartag2014logarithmically}
\leavevmode\vrule height 2pt depth -1.6pt width 23pt, {\em
  Logarithmically-concave moment measures i}, in Geometric Aspects of
  Functional Analysis: Israel Seminar (GAFA) 2011-2013, Springer, 2014,
  pp.~231--260.

\bibitem{kolesnikov2016riemannian}
{\sc A.~V. Kolesnikov and E.~Milman}, {\em Riemannian metrics on convex sets
  with applications to poincar{\'e} and log-sobolev inequalities}, Calculus of
  Variations and Partial Differential Equations, 55 (2016), p.~77.

\bibitem{ledoux2015stein}
{\sc M.~Ledoux, I.~Nourdin, and G.~Peccati}, {\em Stein’s method, logarithmic
  sobolev and transport inequalities}, Geometric and Functional Analysis, 25
  (2015), pp.~256--306.

\bibitem{lelievre2024optimizing}
{\sc T.~Leli{\`e}vre, G.~A. Pavliotis, G.~Robin, R.~Santet, and G.~Stoltz},
  {\em Optimizing the diffusion coefficient of overdamped langevin dynamics},
  arXiv preprint arXiv:2404.12087,  (2024).

\bibitem{li2024sharp}
{\sc M.~T. Li, T.~Cui, F.~Li, Y.~Marzouk, and O.~Zahm}, {\em Sharp detection of
  low-dimensional structure in probability measures via dimensional logarithmic
  sobolev inequalities}, arXiv preprint arXiv:2406.13036,  (2024).

\bibitem{li2023principal}
{\sc M.~T. Li, Y.~Marzouk, and O.~Zahm}, {\em Principal feature detection via
  $\phi$-sobolev inequalities}, Bernoulli, 30 (2024), pp.~2979--3003.

\bibitem{li2022mirror}
{\sc R.~Li, M.~Tao, S.~S. Vempala, and A.~Wibisono}, {\em The mirror langevin
  algorithm converges with vanishing bias}, in International Conference on
  Algorithmic Learning Theory, PMLR, 2022, pp.~718--742.

\bibitem{liu2023mirror}
{\sc G.-H. Liu, T.~Chen, E.~Theodorou, and M.~Tao}, {\em Mirror diffusion
  models for constrained and watermarked generation}, Advances in Neural
  Information Processing Systems, 36 (2023), pp.~42898--42917.

\bibitem{liu2016kernelized}
{\sc Q.~Liu, J.~Lee, and M.~Jordan}, {\em A kernelized stein discrepancy for
  goodness-of-fit tests}, in International conference on machine learning,
  PMLR, 2016, pp.~276--284.

\bibitem{nakai2004density}
{\sc E.~Nakai, N.~Tomita, and K.~Yabuta}, {\em Density of the set of all
  infinitely differentiable functions with compact support in weighted sobolev
  spaces}, Scientiae Mathematicae Japonicae, 60 (2004), pp.~121--128.

\bibitem{nesterov1983method}
{\sc Y.~E. Nesterov}, {\em A method of solving a convex programming problem
  with convergence rate $\mathcal o(k^2)$}, in Doklady Akademii Nauk, vol.~269,
  Russian Academy of Sciences, 1983, pp.~543--547.

\bibitem{Teixeira2020}
{\sc M.~T. Parente, J.~Wallin, and B.~Wohlmuth}, {\em Generalized bounds for
  active subspaces}, Electronic Journal of Statistics, 14 (2020), pp.~917 --
  943.

\bibitem{patterson2013stochastic}
{\sc S.~Patterson and Y.~W. Teh}, {\em Stochastic gradient riemannian langevin
  dynamics on the probability simplex}, Advances in neural information
  processing systems, 26 (2013).

\bibitem{payne1960optimal}
{\sc L.~E. Payne and H.~F. Weinberger}, {\em An optimal poincar{\'e} inequality
  for convex domains}, Archive for Rational Mechanics and Analysis, 5 (1960),
  pp.~286--292.

\bibitem{pillaud2023kernelized}
{\sc L.~Pillaud-Vivien and F.~Bach}, {\em Kernelized diffusion maps}, in The
  Thirty Sixth Annual Conference on Learning Theory, PMLR, 2023,
  pp.~5236--5259.

\bibitem{pillaud2020statistical}
{\sc L.~Pillaud-Vivien, F.~Bach, T.~Leli{\`e}vre, A.~Rudi, and G.~Stoltz}, {\em
  Statistical estimation of the poincar{\'e} constant and application to
  sampling multimodal distributions}, in International Conference on Artificial
  Intelligence and Statistics, PMLR, 2020, pp.~2753--2763.

\bibitem{poincare1890equations}
{\sc H.~Poincar{\'e}}, {\em Sur les {\'e}quations aux d{\'e}riv{\'e}es
  partielles de la physique math{\'e}matique}, American Journal of Mathematics,
   (1890), pp.~211--294.

\bibitem{qian1999momentum}
{\sc N.~Qian}, {\em On the momentum term in gradient descent learning
  algorithms}, Neural networks, 12 (1999), pp.~145--151.

\bibitem{rockafellar1970convex}
{\sc R.~T. Rockafellar}, {\em Convex Analysis}, Princeton University Press,
  Princeton, 1970.

\bibitem{roustant2017poincare}
{\sc O.~Roustant, F.~Barthe, and B.~Iooss}, {\em Poincar{\'e} inequalities on
  intervals--application to sensitivity analysis},  (2017).

\bibitem{santambrogio2016dealing}
{\sc F.~Santambrogio}, {\em Dealing with moment measures via entropy and
  optimal transport}, Journal of Functional Analysis, 271 (2016), pp.~418--436.

\bibitem{saumard2019weighted}
{\sc A.~Saumard}, {\em Weighted poincar{\'e} inequalities, concentration
  inequalities and tail bounds related to stein kernels in dimension one},
  Bernoulli, 25 (2019), pp.~3978--4006.

\bibitem{song2019derivative}
{\sc S.~Song, T.~Zhou, L.~Wang, S.~Kucherenko, and Z.~Lu}, {\em
  Derivative-based new upper bound of sobol’sensitivity measure}, Reliability
  Engineering \& System Safety, 187 (2019), pp.~142--148.

\bibitem{stein1986approximate}
{\sc C.~Stein}, {\em Approximate computation of expectations}, IMS, 1986.

\bibitem{sutskever2013importance}
{\sc I.~Sutskever, J.~Martens, G.~Dahl, and G.~Hinton}, {\em On the importance
  of initialization and momentum in deep learning}, in International conference
  on machine learning, PMLR, 2013, pp.~1139--1147.

\bibitem{vempala2019rapid}
{\sc S.~Vempala and A.~Wibisono}, {\em Rapid convergence of the unadjusted
  langevin algorithm: Isoperimetry suffices}, Advances in neural information
  processing systems, 32 (2019).

\bibitem{verdiere2023diffeomorphism}
{\sc R.~Verdi{\`e}re, C.~Prieur, and O.~Zahm}, {\em Diffeomorphism-based
  feature learning using poincar{\'e} inequalities on augmented input space},
  (2023).

\bibitem{villani2021topics}
{\sc C.~Villani}, {\em Topics in optimal transportation}, Graduate Studies in
  Mathematics,  (2003).

\bibitem{zalinescu2002convex}
{\sc C.~Zalinescu}, {\em Convex analysis in general vector spaces}, World
  scientific, 2002.

\bibitem{zienkiewicz2000finite}
{\sc O.~C. Zienkiewicz, R.~L. Taylor, and R.~L. Taylor}, {\em The finite
  element method: solid mechanics}, vol.~2, Butterworth-heinemann, 2000.

\end{thebibliography}
}

\end{document}